\documentclass[10pt, a4paper]{article} 
\usepackage[a4paper, left=3cm, right=3cm, top=3.5cm, bottom=3.5cm]{geometry}
\usepackage{amsmath}
\usepackage{amssymb}
\usepackage{mathrsfs}
\usepackage{amsthm}
\usepackage[english]{babel}
\usepackage{xcolor}
\usepackage{verbatim}
\usepackage{stmaryrd}
\usepackage{hyperref}
\usepackage[nameinlink,capitalize]{cleveref}
\usepackage{cancel}
\usepackage{slashed}
\usepackage{centernot}
\usepackage{todonotes}
\usepackage{tikz}

\usepackage[
  style=alphabetic,
  maxnames=99,
  doi=true,
  isbn=false,
  url=false,
  eprint=false
]{biblatex}
\newcommand{\R}{\mathbb{R}}
\newcommand{\E}{\mathbb{E}}

\newcommand{\esssup}{\mathop{\mathrm{ess\,sup}}}

\theoremstyle{definition}
\newtheorem{defn}{Definition}[section]
\newtheorem{ex}[defn]{Example}
\newtheorem{rmk}[defn]{Remark}
\newtheorem*{notation}{Notation}
\newtheorem{assumptions}[defn]{Assumptions}
\theoremstyle{plain}
\newtheorem{thm}[defn]{Theorem}
\newtheorem*{thm*}{Theorem}
\newtheorem{lemma}[defn]{Lemma}

\newtheorem{prop}[defn]{Proposition}
\newtheorem*{teoremasingolo}{Theorem}
\newtheorem*{fact}{Fact}

\newtheorem*{assum*}{Assumptions}

\crefname{defn}{Definition}{Definitions}
\crefname{ex}{Example}{Examples}
\crefname{rmk}{Remark}{Remarks}
\crefname{assumptions}{Assumption}{Assumptions}
\crefname{thm}{Theorem}{Theorems}
\crefname{lemma}{Lemma}{Lemmas}
\crefname{coroll}{Corollary}{Corollaries}
\crefname{prop}{Proposition}{Propositions}
\crefname{estimate}{Estimate}{Estimates}

\title{Nonlinear rough Fokker--Planck equations}
\author{Fabio Bugini\thanks{Technische Universit\"at Berlin (fabio.bugini@tu-berlin.de)} , Peter K. Friz\thanks{Technische Universit\"at Berlin and Weierstraß-Institut (friz@math.tu-berlin.de)} , Wilhelm Stannat\thanks{Technische Universit\"at Berlin (stannat@math.tu-berlin.de)}}

\begin{document}

\maketitle

\begin{abstract}
McKean–-Vlasov SDEs describe systems where the dynamics depend on the law of the process. The corresponding Fokker--Planck equation is a nonlinear, nonlocal PDE for the corresponding measure flow. In the presence of common noise and conditional law dependence, the evolution becomes random and is governed by a stochastic Fokker–-Planck equation; that is, a nonlinear, nonlocal SPDE in the space of measures. (Such equations constitute an important ingredient in the theory of mean-field games with common noise.) Well-posedness of such SPDEs is a difficult problem, the best result to date due to Coghi -- Gess (2019), which however comes with dimension-dependent regularity assumptions. In the present work, we show how rough path techniques can circumvent these entirely. Hence, and somewhat contrarily to common belief, the use of rough paths leads to substantially less regularity demands on the coefficients than methods rooted in classical stochastic analysis methods. 
\end{abstract}



\tableofcontents

\section{Introduction}

In this paper we present an existence and uniqueness result for measure-valued Fokker--Planck equations of the form
\begin{multline} \label{eq:meanfieldroughPDEs_intro}
        d\mu_t = \Bigg(\frac{1}{2} \sum_{i,j=1}^d \frac{\partial^2}{\partial x^j \partial x^i} (a^{ij}(t,x,\mu_t) \mu_t) - \sum_{i=1}^d \frac{\partial}{\partial x^i} (b^i(t,x,\mu_t)\mu_t) \Bigg) \, dt  \\ - \sum_{\kappa=1}^n  \sum_{i=1}^d \frac{\partial}{\partial x^i} (f_\kappa^i(t,x,\mu_t) \mu_t) \, dW^\kappa_t
\end{multline}
on a finite time interval $[0,T]$.
The solution is a curve $\mu:[0,T] \to \mathcal{P}(\R^d)$ in the space of probability measures on $\R^d$ and $W:[0,T] \to \R^n$ is an $\alpha$-H\"older continuous path, in the rough regime $\alpha \in (1/3,1/2]$. The correct rough path formulation is left to \cref{def:solution_meanfieldroughPDE} and notably accomodates Brownian randomisation $W \rightsquigarrow B^0 (\omega)$, thereby covering the popular situation {corresponding to} McKean--Vlasov SDEs with common noise, cf. \eqref{eq:McKeanVlasovSDE_intro} below.  
The coefficients \begin{equation*}
    (a,b,f): [0,T] \times \R^d \times \mathcal{P}(\R^d) \to \R^{d \times d} \times \R^d \times \R^{d \times n} 
\end{equation*} 
also depend on the solution itself, potentially in a nonlinear manner. Equations of the form \eqref{eq:meanfieldroughPDEs_intro} are often called nonlinear and nonlocal.  
For any $(t,x,\mu) \in [0,T] \times \R^d \times \mathcal{P}(\R^d)$, the ($d\times d$)-matrix $a(t,x,\mu)$ is symmetric and positive semi-definite. 
The main result of our work is the following 

\begin{teoremasingolo}[See \cref{thm:mainresult} below] Let $\nu \in \mathcal{P}(\R^d)$ such that $\int_{\R^d} |x|^2 \, \nu(dx)$ is finite. Then, under natural regularity assumptions on the coefficients $a,b,f$ (see \cref{assumptions_existence} and \cref{assumptions_uniqueness} below), there exists a unique solution $\mu$ to the {\em nonlinear, nonlocal rough Fokker-Planck equation} \eqref{eq:meanfieldroughPDEs_intro}, in the precise sense of \cref{def:solution_meanfieldroughPDE} below, such that $\mu_0=\nu$.      
\end{teoremasingolo} \medskip

\textbf{Motivation. } For sake of better readability let us assume that all the objects introduced in this informal discussion are one-dimensional. 
Let $N \in \mathbb{N}_{\ge 1}$ and let $B,B^0,B^1,\dots,B^N$ be independent Brownian motions on some probability space $(\Omega,\mathcal{F},\mathbb{P})$. 
Let us consider a system of $N$ interacting particles with common noise, whose dynamics is described via \begin{equation*}
    \begin{cases}
        dX^{i,N}_t = b_t(X^{i,N}_t, \mu^N_t) \, dt + \sigma_t (X^{i,N}_t, \mu^N_t) \, dB^i_t + f_t (X^{i,N}_t, \mu^N_t) \circ dB^0_t, \ t \in [0,T] \\ \hfill i=1,\dots,N \\
        \mu^N_t = \frac{1}{N} \sum_{j=1}^N \delta_{X^{j,N}_t} .        
    \end{cases}
\end{equation*}
Notice how the position $X^{i,N}$ of every particle also depends on the average position of all the other particles in the system, through the empirical measure $\mu^N$. This interaction is often referred to as being of mean-field type. 
Observe in addition how each one of the particles is influenced by two independent sources of noise: an individual one $B^i$, specific for the particle $X^{i,N}$, and a common noise $B^0$, visible to every particle. 
For what comes later in the paper, it is convenient to think the latter as a Stratonovich noise; the It\^o case can be obtained by changing the drift coefficient $b$, 
according to the It\^o-Stratonovich correction. 
Assume that the initial conditions $\{X^{i,N}_0\}_{i=1,\dots,N}$ are exchangeable, conditionally independent given $B^0$ and identically distributed (denote by $\nu:=\text{Law}(X_0^{1,N})$). 
As the number $N$ of particles grows to $+\infty$, particles become  asymptotically independent, conditional on $B^0$; the effect of the individual noises averages out, but the influence of the common noise is still present. 
At least formally, every $X^{i,N}$ converges (in a suitable sense) to independent copies of $X$, being $X$ the solution to 
\begin{equation} \label{eq:McKeanVlasovSDE_intro}
    \begin{cases}
        dX_t = b_t(X_t, \mu_t) \, dt + \sigma_t (X_t, \mu_t) \, dB_t + f_t (X_t, \mu_t)  \circ dB^0_t, \quad t \in [0,T] \\
        \mu_t = \text{Law}(X_t \mid \mathcal{F}_t^{B^0})  \\
        \text{Law}(X_0) = \nu , 
    \end{cases}
\end{equation}
where $\mathcal{F}_t^{B^0} = \sigma(B^0_r , \ 0 \le r \le t)$.
This phenomenon is known as \textit{conditional propagation of chaos}. For a proper discussion see, for instance, \cite[Section 2.1.2]{CARMONADELARUE_volumeII}. 
Again at a formal level, the curve of random distributions $t \mapsto \mu_t := \text{Law}(X_t \mid \mathcal{F}_t^{B^0})$ is expected to be a solution to a nonlinear Fokker--Planck stochastic partial differential equation of the form \begin{equation} \label{eq:Fokker-PlanckSPDE_intro}
    d\mu_t = \left( \frac{1}{2} \partial^2_x (\sigma_t(x,\mu_t)^2\mu_t) - \partial_x ( b_t(x,\mu_t)  \mu_t) \right) \, dt  - \partial_x (f_t(x,\mu_t) \mu_t) \circ dB^0_t . 
\end{equation}
In this paper, we adopt a rough path point of view to characterize solutions to \eqref{eq:Fokker-PlanckSPDE_intro}.
Almost every trajectory $W = (B^0_t(\omega))_{t \in [0,T]}$ of the common Stratonovich noise can be lifted to a (weakly geometric) rough path $\mathbf{W}=(W,\mathbb{W})$ (see \cref{ex:liftofBrownianmotion} below). 
We can therefore consider the following mixed rough and stochastic differential equation, which naturally arises by freezing a trajectory of the common noise in equation \eqref{eq:McKeanVlasovSDE_intro}:
\begin{equation} \label{eq:McKeanVlasovroughSDE_intro}
    \begin{cases}
        dX^\mathbf{W}_t = b_t(X^\mathbf{W}_t, \mu^\mathbf{W}_t) \, dt + \sigma_t (X^\mathbf{W}_t, \mu^\mathbf{W}_t) \, dB_t + f_t (X^\mathbf{W}_t, \mu^\mathbf{W}_t)  \, d\mathbf{W}_t, \quad t \in [0,T] \\
        \mu^\mathbf{W}_t = \text{Law}(X^\mathbf{W}_t)  \\
        \text{Law}(X^\mathbf{W}_0) = \nu . 
    \end{cases}
\end{equation}
This kind of equations has been introduced and studied in \cite{FHL25}. We also give a short overview of their work in \cref{section:McKean--ValsovroughSDEs} of the present paper.  
The aim of our work is therefore to characterize, for any weakly geometric rough path $\mathbf{W}=(W,\mathbb{W})$, the curve $t \mapsto \mu_t^\mathbf{W}:=\text{Law}(X_t^\mathbf{W})$ as the unique solution to a measure-valued nonlinear rough partial differential equation (rough PDE) of the form \begin{equation*}
    d\mu^\mathbf{W}_t = \left( \frac{1}{2} \partial^2_x (\sigma_t(x,\mu^\mathbf{W}_t)^2 \mu^\mathbf{W}_t) - \partial_x (b_t(x,\mu^\mathbf{W}_t) \mu^\mathbf{W}_t) \right) \, dt - \partial_x (f_t(x,\mu^\mathbf{W}_t) \mu^\mathbf{W}_t) \, d\mathbf{W}_t. 
\end{equation*} 
As shown in \cite[Section 3.3]{FHL25}, also \cite[Section 5]{FLZ25}, the law of $X^\mathbf{W}_t$ yields a regular conditional distribution of the conditional law of $X_t$ given $(B^0_r, \ r \in [0,t])$, provided $\mathbf{W}$ is taken to be the Stratonovich rough path lift of $B^0$ (cf. \cref{ex:liftofBrownianmotion}). More precisely, \begin{equation*}
    \mu_t (\omega) := \text{Law}(X_t \mid \mathcal{F}^{B^0}_t) (\omega) = \mu_t^\mathbf{W} \big|_{\mathbf{W}=\mathbf{B}^{\text{Strato}}(\omega)}
\end{equation*}
for $\mathbb{P}$-almost every $\omega \in \Omega$. In particular, the objects we are concerned with in this paper are the ``right'' ones, in the sense that they coincide with objects previously treated from within It\^o calculus. Remarkably enough, the usual perception that rough path requires more regularity than It\^o based methods does not apply here: stochastic (state of the art) methods \cite{CG19} require dimension-dependent regularity assumptions, which is not at all the methods presented here i.e. we obtain well-posedness under dimension-independent regularity assumptions. 
We also note that our rough path approach to common noise goes immediately beyond Brownian (or semimartingale) common noise, and indeed provides a satisfactory answer to the question why one should prescribe statistics for the common noise at the first place, given that one wishes to work conditionally under the common noise. \\


\textbf{Literature review. }
Well-posedness of equation \eqref{eq:meanfieldroughPDEs_intro} has been investigated in \cite{CG19} for the case where $W$ is an Itô Brownian motion (i.e., not a deterministic path).
\noindent The authors adopt a stochastic PDEs framework and establish uniqueness through a duality argument. However, their results and assumptions inherently depend on the dimension of the state space, partly due to Sobolev embeddings and other technical considerations.
(Specifically, \cite[Assumptions 5.1]{CG19} require the coefficients to have $m$ degrees of regularity, with $m>d/2+2$ where $d$ is the dimension of the state space. ) 
In \cite{CN21}, the authors are concerned with the study of equation \eqref{eq:meanfieldroughPDEs_intro} in a rough path framework. They address the challenges posed by the measure-dependent coefficient $f$ in the rough term by assuming a linear dependence on the measure. 

Central to our approach are rough stochastic differential equations \cite{FHL21, FHL25}, which ultimately allow us to accommodate, under natural regularity assumptions, the desirable general (nonlinear) dependencies in the measure, as in \cite{CG19} and required in applications \cite{CARMONADELARUE_volumeI, CARMONADELARUE_volumeII}.



In our work, we integrate rough path theory with analysis on the space of probability measures, inspired by Lions' differential calculus in measure spaces. We also draw inspiration from \cite{BCD19}, although their framework is different, as their focus is on the mean-field rough differential equations driven by arbitrary random rough paths, which (in this generality) does not come with accompanying stochastic partial differential equations.  
\\

\textbf{Structure of the paper. } In \cref{section:preliminaries}, we present some preliminaries which can be useful in understanding our work. 
We start with a very short introduction on rough path theory;
we then introduce the basics of Lions' differential calculus and present some results about Lions derivatives; we also recall the concept of differentiability along sub-Banach spaces, as introduced in \cite{FHL25}. 
When it comes to proving well-posedness of \eqref{eq:meanfieldroughPDEs_intro}, the two main tools are: a general theory on McKean--Vlasov rough SDEs (introduced in \cite{FHL25}; an overview is given in \cref{section:McKean--ValsovroughSDEs}), and a uniqueness result for the measure-independent linear equation (given in \cref{section:linearroughPDEs}; based on existence results for the dual function-valued backward equation studied in \cite{BFS24}).
In \cref{section:solvingmeanfieldroughPDEs}, we formalize the discussion on equation \eqref{eq:meanfieldroughPDEs_intro} by providing a precise definition of its solutions. 
The main results of our work, namely some existence and uniqueness statements for \eqref{eq:meanfieldroughPDEs_intro}, are given in \cref{section:existence} and \cref{section:uniqueness}, respectively. \\

\textbf{Frequently used notation. } Let $(V,|\cdot|_V)$ and $(\bar V,|\cdot|_{\bar V})$ be two Banach spaces. 
The Borel $\sigma$-algebra on $V$ is denoted by $\mathcal{B}(V)$. Unless otherwise specified, all Banach spaces are assumed to be endowed with their Borel $\sigma$-algebra.
If clear from the context, $|\cdot|_V$ and $|\cdot|_{\bar V}$ are simply denoted by $|\cdot|$. 
The tensor product of $V$ and $\bar V$ is denoted by $V \otimes \bar V$ {and is endowed with the projective topology}. 
Let $k \in \mathbb{N}_{\ge 1}$.
We denote by $\mathscr{L}^k(V,\bar V)$ the space of multilinear and bounded operators from $V^k$ to $\bar V$, and $\mathscr{L}:=\mathscr{L}^1$.
The norm on $\mathscr{L}^k(V,\bar V)$ is defined by $|\phi| := \sup_{v \in V^k, \sum_{i=1}^k|v^i|_V \le 1} |\phi(v)|_{\bar V}$ for any $\phi \in \mathscr{L}^k(V,\bar V)$.
The following identification holds: $\mathscr{L}^k(V,\bar V) \equiv \mathscr{L}(V^{\otimes k},\bar V)$. We denote by $\mathrm{Bd}(V;\bar V)$ the space of bounded functions from $V$ to $\bar V$. 

Given a path $Y:[0,T] \to V$, we denote by $\delta Y_{s,t} := Y_t - Y_s$ its increment on the time-interval $[s,t]$. Given $a, b : [0,T] \times [0,T] \to V$, we write $a_{s,t}\overset{\gamma}{=}b_{s,t}$ if there is a positive constant $C>0$ independent of $s,t$ and such that $|a_{s,t} - b_{s,t}| \le C |t-s|^\gamma$, for any $s,t \in [0,T]$ and for some $\gamma \in (0,1]$. 
In particular, $g:[0,T] \to V$ is $\gamma$-H\"older continuous if and only if $g_t-g_s \overset{\gamma}{=} 0$. 

Given a vector $x \in \R^d$, $x^i$ denotes its $i$-th component. Similarly, $g^i$ denotes the $i$-th component of the map $g : V \to \R^d$. The minimum between $x,y \in \R$ is denoted by $x \wedge y$.

Let $(\Omega,\mathcal{F},\mathbb{P})$ be a probability space. The law of a $V$-valued random variable $X$ is denoted by $\mathcal{L}_X$, and it is by construction a Borel probability measure on $V$. 
Measurability of mappings $(\Omega,\mathcal{F}) \to (V,\mathcal{B}(V))$ is always to be understood in strong sense.
Given $p \in [1,\infty)$, we denote by $L_p(\Omega,\mathbb{P};V)$ - or simply by $L_p(\Omega;V)$ - the space of $V$-valued random variables $X$ on $(\Omega,\mathcal{F},\mathbb{P})$ such that $\|X\|_p := \E (|X|^p)^\frac{1}{p} < +\infty$, where $\E(\cdot)$ denotes the expectation under $\mathbb{P}$. 
We denote $\|X\|_\infty := \inf\{M \in [0,\infty] \mid |X|\le M \  \mathbb{P}\text{-a.s.}\}$. 
We often identify the spaces $L_2(\Omega;\R^d \otimes \R^d) \equiv L_2(\Omega;\R^d) \otimes L_2(\Omega;\R^d)$. \\

\section{Preliminaries} \label{section:preliminaries}
\subsection{Rough path theory}
For a detailed discussion on rough paths (and more) see \cite{FRIZHAIRER}. 
Let $T > 0$. 

    \begin{defn} \label{def:roughpath} (see \cite[Definition 2.1]{FRIZHAIRER})
    Let $\alpha \in (\frac{1}{3},\frac{1}{2}]$.
    A \textit{rough path} on $[0,T]$ with values in $\R^n$ is a pair $\mathbf{W}=(W,\mathbb{W})$ such that \begin{enumerate} \renewcommand{\labelenumi}{(\alph{enumi})}
    \item $W:[0,T] \to \R^n$ with $\delta W_{s,t} \overset{\alpha}{=} 0$ ;
    \item $\mathbb{W}:[0,T]^2 \to \R^n \otimes \R^n$ with $\mathbb{W}_{s,t} \overset{2\alpha}{=} 0$;
    \item (Chen's relation) for any $s,u,t \in [0,T]$, it holds \begin{equation} \label{eq:chenrelation}
        \mathbb{W}_{s,t}-\mathbb{W}_{s,u}-\mathbb{W}_{u,t} = \delta W_{s,u}\otimes\delta W_{u,t}
    \end{equation} or, equivalently, $\mathbb{W}_{s,t}^{ij}-\mathbb{W}_{s,u}^{ij}-\mathbb{W}_{u,t}^{ij} = \delta W_{s,u}^i \delta W_{u,t}^j$ for any $i,j=1,\dots,n$.
\end{enumerate}
We write $\mathbf{W}\in \mathscr{C}^\alpha([0,T];\R^n)$.
\end{defn}

\begin{defn} (see \cite[ Section 2.2]{FRIZHAIRER})
    A rough path $\mathbf{W}=(W,\mathbb{W})$ is said to be \textit{weakly geometric} if, for any $s,t \in [0,T]$, \begin{equation*}
    Sym(\mathbb{W}_{s,t}):=\frac{\mathbb{W}_{s,t}+\mathbb{W}_{s,t}^\top}{2} = \frac{1}{2}\delta W_{s,t} \otimes \delta W_{s,t}.
    \end{equation*}
    or, equivalently, $\frac{1}{2}(\mathbb{W}_{s,t}^{ij}+\mathbb{W}_{s,t}^{ji}) = \frac{1}{2}\delta W_{s,t}^i \delta W_{s,t}^j$ for any $i,j=1,\dots,n$. 
    We write $\mathbf{W}\in \mathscr{C}^\alpha_g([0,T];\R^n)$.
\end{defn}

\begin{rmk}
    Let $W:[0,T] \to \R^n$ be a smooth path. For any $t \in [0,T]$, the integral $\int_0^t Y_r \, dW_r$ is well-defined in $\R^d$ as a Riemann--Stieltjes integral, provided that $Y :[0,T] \to \mathscr{L}(\R^n,\R^d)$ is, for example, continuous. 
    Define \begin{equation*}
        \mathbb{W}_{s,t} = (\mathbb{W}_{s,t}^{ij})_{i,j=1,\dots,n} := \left(\int_s^t (W_r^i - W_s^i) \, dW^j_r \right)_{i,j=1,\dots,n}  
    \end{equation*}
    for any $s,t \in [0,T]$, where the integral is understood as a Riemann--Stieltjes integral.
    It is straightforward to verify that $\mathbf{W} =(W,\mathbb{W}) \in \mathscr{C}^\alpha_g([0,T];\R^n)$ for any $\alpha \in (\frac{1}{3}, \frac{1}{2}]$. 
    When $W$ is only $\alpha$-H\"older continuous, the previous Riemann--Stieltjes integral may no longer be well-defined. In such a case, we can think of $\mathbb{W}^{ij}$ as the postulated value of this integral.
    Rough path theory provides a framework to extend Riemann--Stieltjes integration with respect to $dW$ to a larger class of integrands $Y$. 
\end{rmk}

\begin{ex}[Rough path lift of a Brownian motion] \label{ex:liftofBrownianmotion}
     Let $B=(B_t)_{t \in [0,T]}$ be an $n$-dimensional Brownian motion on $(\Omega,\mathcal{F},\mathbb{P})$. 
     For any $s,t \in [0,T]$, define its iterated (Stratonovich) integral \begin{equation*}
         \mathbb{B}_{s,t}^{\text{Strato}} := \int_s^t \delta B_{s,r} \otimes \circ dB_r .
     \end{equation*}
     For definiteness, $(\mathbb{B}_{s,t}^{\text{Strato}})^{ij} (\omega) = (\int_s^t \delta B^i_{s,r} \circ dB^j_r)(\omega)$ for any $i,j=1,\dots,n$ and for any $\omega \in \Omega$. Then, for any $\alpha \in (\frac{1}{3},\frac{1}{2})$ and $\mathbb{P}$-almost surely, \begin{equation*}
         \mathbf{B}^{\text{Strato}} = (B,\mathbb{B}^{\text{Strato}}) \in \mathscr{C}^\alpha_g ([0,T];\R^n) 
     \end{equation*}
     (see \cite[Proposition 3.5]{FRIZHAIRER}). 
     The It\^o lift of $B$ is defined analogously, setting $\mathbb{B}_{s,t}^{\text{It\^o}}:= \int_s^t \delta B_{s,r} \otimes dB_r$. However, $\mathbf{B}^{\text{It\^o}} = (B,\mathbb{B}^{\text{It\^o}})$ is not weakly geometric. The It\^o-Stratonovich correction still holds: $\mathbb{B}_{s,t}^{\text{Strato}} = \mathbb{B}_{s,t}^{\text{It\^o}} + \frac{1}{2} (t-s) I$, being $I$ the identity element of $\R^n \otimes \R^n$. 
\end{ex}

\subsection{Lions derivatives}
For a complete overview on Lions' calculus see \cite[Section 5]{CARMONADELARUE_volumeI}. 
Let $d \in \mathbb{N}_{\ge 1}$. We denote by $\mathcal{P}(\R^d)$ the set of Borel probability measures on $\R^d$.
Let $\mu,\nu \in \mathcal{P}(\R^d)$.
A coupling of $\mu$ and $\nu$ is a Borel probability measure $\pi$ on $\R^d \times \R^d$ such that, for any $A \in \mathcal{B}(\R^d)$, $\pi(A \times \R^d) = \mu(A)$ and $\pi(\R^d \times A) = \nu (A)$.  
Let $p \in [1,\infty)$. 
The set \begin{equation*}
    \mathcal{P}_p(\R^d) := \Big\{ \mu \in \mathcal{P}(\R^d) \mid \int_{\R^d} |z|^p \, \mu(dz) < +\infty \Big\}
\end{equation*}
can be endowed with the so-called $p$-Wasserstein distance $\mathcal{W}_p$, defined by \begin{equation*}
    \mathcal{W}_p(\mu,\nu) := \inf \left\{ \left(\int_{\R^d \times \R^d} |z-w|_{\R^d}^p \, \pi(dz,dw) \right)^\frac{1}{p} \mid \text{$\pi$ is a coupling of $\mu$ and $\nu$}  \right\} . 
\end{equation*} 
In particular, for any probability space $(\Omega,\mathcal{F},\mathbb{P})$ and for any pair $X,Y$ of random variables on $(\Omega,\mathcal{F},\mathbb{P})$ such that $\mathcal{L}_X = \mu$ and $\mathcal{L}_Y = \nu$, \begin{equation*}
    \mathcal{W}_p(\mu,\nu) \le \|X-Y\|_p .
\end{equation*}
\noindent By H\"older's inequality, for $p \le q$ it holds $\mathcal{W}_p(\mu,\nu) \le \mathcal{W}_q(\mu,\nu)$.  It is also possible to prove that $(\mathcal{P}_p(\R^d), \mathcal{W}_p)$ is a complete separable metric space. Moreover, the Borel $\sigma$-algebra of $\mathcal{P}_p(\R^d)$ is generated by the family of mappings $\{\mathcal{P}_p(\R^d) \ni \mu \mapsto \mu(A) \mid A \in \mathcal{B}(\R^d)\}$. In the rest of the paper we focus on the case $p=2$. \\

\noindent Let $g:\mathcal{P}_2(\R^d) \to \R$ and let $(\Omega,\mathcal{F},\mathbb{P})$ be any (atomless Polish) probability space. 
The Lions lift of $g$ is the map $\hat{g}: L_2(\Omega;\R^d) \to \R$ defined by $\hat{g}(X) := g(\mathcal{L}_X)$.
\begin{defn}
    We say that $g$ is \textit{(continuously) L-differentiable} at $\mu_0 \in \mathcal{P}_2(\R^d)$ if there exists a random variable $X_0$ on $(\Omega,\mathcal{F},\mathbb{P})$ with $\mathcal{L}_{X_0} = \mu_0$ and such that the lifted map $\hat{g}$ is (continuously) Fréchet differentiable at $X_0$.
\end{defn}

{
\noindent The Fréchet derivative of $\hat{g}$ at $X_0$ evaluated in the direction $Y \in L_2(\Omega;\R^d)$ is denoted by $D\hat{g}(X_0)[Y]$. In particular, the linear operator $D \hat g (X_0) = D \hat g (X_0) [\, \cdot \,] \in \mathscr{L}(L_2(\Omega;\R^d) , \R)$ can be seen as an element of $L_2(\Omega;\R^d)$, due to the identification of  $L_2(\Omega;\R^d)$ with its dual. 
}
It is possible to prove that this notion of differentiability is intrinsic, i.e.\ it does not depend upon the choice of the random variable whose law is $\mu_0$ nor on the probability space (cf.\ \cite[Proposition 5.24]{CARMONADELARUE_volumeI}). We say that $g$ is (continuously) L-differentiable if it is (continuously) L-differentiable at any $\mu_0 \in \mathcal{P}_2(\R^d)$. \\

\noindent Let $g:\mathcal{P}_2(\R^d) \to \R$ be continuously L-differentiable. Then the following hold:
    \begin{enumerate}\renewcommand{\labelenumi}{(\roman{enumi})}
        \item for any $\mu \in \mathcal{P}_2(\R^d)$, there exists $\partial_\mu g(\mu)(\cdot) \in L_2(\R^d,\mu;\R^d)$ such that, for any $X \in L_2(\Omega;\R^d)$ with $\mathcal{L}_X=\mu$, 
            \begin{equation*}
                D\hat{g}(X) = \partial_\mu g (\mu) (X) \qquad \text{$\mathbb{P}$-a.s.}, 
            \end{equation*}
        i.e.\ $D\hat{g}(X)(\omega) = \partial_\mu g (\mu) (X(\omega))$ for $\mathbb{P}$-almost every $\omega \in \Omega$ (cf.\ \cite[Proposition 5.25]{CARMONADELARUE_volumeI}). In particular, for any $Y \in L_2(\Omega;\R^d)$, we have that $D\hat{g}(X)[Y] = \E(\partial_\mu g(\mu)(X) \cdot Y)$, where $\cdot$ denotes the scalar product in $\R^d$  ;
        \item $\partial_\mu g(\mu)(v)$ has a jointly measurable $\mu$-version, namely it is possible to redefine $\partial_\mu g(\mu)$ on a $\mu$-negligible set in such a way that the mapping $\mathcal{P}_2(\R^d) \times \R^d \ni (\mu,v) \mapsto \partial_\mu g(\mu) (v) \in \R^d$ is jointly Borel measurable (cf.\ \cite[Proposition 5.33]{CARMONADELARUE_volumeI}). If not otherwise specified, in this paper we are always considering the jointly measurable version when dealing with L-derivatives ;
        \item for any $\mu,\nu \in \mathcal{P}_2(\R^d)$ \begin{equation*}
            g(\nu) - g(\mu) = \int_{\R^d \times \R^d} \partial_\mu g (\mu) (v) \cdot (w-v) \, \pi(dv,dw) + \vartheta_{\mu,\nu}(\pi) ,
        \end{equation*}
        where $\pi$ is any coupling of $\mu$ and $\nu$ and $$\vartheta_{\mu,\nu}(\pi) = o\left( \Big(\int_{\R^d \times \R^d} |w-v|^2 \, \pi(dv,dw) \Big)^\frac{1}{2}\right)$$ as $\int_{\R^d \times \R^d} |w-v|^2 \, \pi(dv,dw) \to 0$ (cf.\ \cite[Lemma 5.30]{CARMONADELARUE_volumeI}). In particular, given any random variables $X,Y$ on $(\Omega,\mathcal{F},\mathbb{P})$ with $\mathcal{L}_X=\mu$ and $\mathcal{L}_Y=\nu$, \begin{equation*}
            g(\nu) - g(\mu) = \E(\partial_\mu g(\mu) (X) \cdot (Y-X)) + \vartheta_{\mu,\nu}(\mathcal{L}_{(X,Y)}). 
        \end{equation*}  
    \end{enumerate} 

    \noindent We say that $g$ is Lipschitz continuous if there exists a positive constant denoted by $|g|_{Lip}$ such that \begin{equation*} |g(\mu) - g(\nu) | \le |g|_{Lip} \mathcal{W}_2(\mu,\nu) \qquad \text{for any $\mu,\nu \in \mathcal{P}_2(\R^d)$} . \end{equation*}
    {We say that $\partial_\mu g: \mathcal{P}_2(\R^d) \times \R^d \to \R^d$ is Lipschitz continuous if there exists a constant $|\partial_\mu g|_{Lip} >0$ such that \begin{equation*}
            |\partial_\mu g(\mu)(v) - \partial_\mu g(\nu)(v')| \le |\partial_\mu g|_{Lip} ( \mathcal{W}_2(\mu,\nu) + |v-v'| ) \qquad \text{for any $\mu,\nu \in \mathcal{P}_2(\R^d)$ and $v,v' \in \R^d$}. 
        \end{equation*} 
    }

    \begin{lemma} \label{lemma:remainderinTaylorexpansion}
        Let $g: \mathcal{P}_2(\R^d) \to \R$ be continuously L-differentiable and assume its L-derivative has a $\mu$-version $\partial_\mu g: \mathcal{P}_2(\R^d) \times \R^d \to \R^d$ which is Lipschitz continuous. Then, for any $\mu,\nu \in \mathcal{P}_2(\R^d)$ and for any coupling $\pi$ of $\mu$ and $\nu$, 
            \begin{equation*} 
            | \vartheta_{\mu,\nu} (\pi) | \le  |\partial_\mu g|_{Lip}  \int_{\R^d \times \R^d} |w-v|^2 \, \pi(dv, dw) .
        \end{equation*}
    \end{lemma}

    \begin{proof}
        Let $\mu,\nu \in \mathcal{P}_2(\R^d)$ be fixed and let $X,Y$ be two arbitrary $\R^d$-valued random variables with $\mathcal{L}_X=\mu$ and $\mathcal{L}_Y = \nu$. 
        Being $\hat{g}:L_2(\Omega;\R^d) \to \R$ continuously Fréchet differentiable, by means of the mean value theorem we can write \begin{align*}
            g(\nu) - g(\mu) &= \hat{g}(Y) - \hat{g}(X) = \int_0^1 D\hat{g}(X + \theta(Y-X)) [Y-X] \, d\theta  \\ 
            &= \int_0^1 \E(\partial_\mu g(\mathcal{L}_{X + \theta(Y-X)}) (X + \theta(Y-X)) \cdot (Y-X) ) \, d\theta .
        \end{align*}
        Let $\pi$ be a coupling of $\mu$ and $\nu$. Then 
        \begin{equation*} 
            \begin{split}
                |\vartheta_{\mu,\nu} ( \pi)| &= |\hat{g}(Y) - \hat{g}(X) - \E(\partial_\mu g(\mu)(X) \cdot (Y-X))| \\
                &\le \int_0^1 \E(|(\partial_\mu g(\mathcal{L}_{X + \theta(Y-X)}) (X + \theta(Y-X)) - \partial_\mu g(\mu) (X) )\cdot (Y-X)|) \, d\theta \\
                &\le |\partial_\mu g|_{Lip} \int_0^1 \E((\mathcal{W}_2(\mathcal{L}_{X + \theta (Y-X)},\mu) + | \theta (Y-X)|)^2)^\frac12 \, d\theta \  \|Y-X\|_2 \\
                &\le |\partial_\mu g|_{Lip} \int_0^1 (2 \mathcal{W}_2(\mathcal{L}_{X + \theta (Y-X)},\mu)^2 + 2\theta^2 \E(|Y-X|^2))^\frac12 \, d\theta \ \|Y-X\|_2 \\
                &\le |\partial_\mu g|_{Lip} \|Y-X\|_2^2 =  |\partial_\mu g|_{Lip}  \int_{\R^d \times \R^d} |w-v|^2 \, \pi(dv,dw),
            \end{split}
        \end{equation*}
        where we used the fact that $\mathcal{W}_2(\mathcal{L}_{X + \theta (Y-X)},\mu) \le \| (X + \theta (Y-X)) - X\|_2 = \theta \|Y-X\|_2 $.
    \end{proof}

    \noindent Let $e \in \mathbb{N}_{\ge 1}$ and let $g=(g^i)_{i=1,\dots,e}:\mathcal{P}_2(\R^d) \to \R^e$. Assuming every $g^i$ is continuously L-differentiable, we define the map $\partial_\mu g:\mathcal{P}_2(\R^d) \times \R^d \to \R^e \otimes \R^d$ by \begin{equation*} 
        (\partial_\mu g(\mu)(v))^{ij} := [\partial_\mu g^i (\mu) (v) ]^j \qquad i=1,\dots,e, \, j=1,\dots,d .  
    \end{equation*}



\subsection{Differentiability along sub-Banach spaces}
We report here the content of \cite[Section 2.2]{FHL25} on differentiability along sub-Banach spaces. 
{To motivate this notion, one can consider the following toy model of a McKean--Vlasov rough SDE in dimension $d=1$:
        \begin{equation} \label{eq:toyRSDE}
        \begin{cases}
            d X_t(\omega) = f(\mu_t) \, d\mathbf{W}_t, \qquad
            \mu_t = \mathcal{L}_{X_t} \\
            X_0(\omega) = \xi(\omega) \in \R ,
        \end{cases}
        \end{equation}
    where $f:\mathcal{P}_2(\R) \to \R$ is defined as $f(\mu) := \int_{\R} h(x) \, \mu(dx)$ some smooth compactly supported function $h$. 
    Denoting by $g=\hat f:L_2(\Omega;\R) \to \R, \, g(X) = \E(h(X))$ the Lions lift of $f$, equation \eqref{eq:toyRSDE} can be regarded as an infinite-dimensional, $L_2(\Omega;\R)$-valued rough differential equation (RDE) 
    \begin{equation*} 
        \begin{cases}
            d X_t = g(X_t) \, d\mathbf{W}_t \\
            X_0 = \xi \in L_2(\Omega;\R) .
        \end{cases}
        \end{equation*}
    Standard well-posedness would require $g$ to be at least twice continuously differentiable, which cannot be expected in general as explained in \cite[Remark 5.80]{CARMONADELARUE_volumeI}. 
    More details are discussed in \cref{rmk:counterexampleCD} below. 
}

Let $E$ be a topological vector space and let $(\mathcal{X}, |\cdot|_\mathcal{X}),(\mathcal{Y}, |\cdot|_\mathcal{Y})$ be two Banach spaces such that $\mathcal{X}$ is continuously embedded into $E$. 
Later in this paper, we are going to apply the results of this section in the case $E = \R^d \times L_2(\Omega;\R^d)$ and $\mathcal{X}=\R^d \times L_p(\Omega;\R^d)$, for $p \in [2,\infty)$, with $|(x,X)|_\mathcal{X} := |x|_{\R^d} + \|X\|_p$. Let $g : E \to \mathcal{Y}$. We define \begin{equation*}
    |g|_\infty := \sup_{x \in E} |g(x)|_{\mathcal{Y}} .
\end{equation*}

\begin{defn} \label{def:HoldercontinuityalongsubBanachspaces}
    Let $\gamma \in (0,1]$. We say that $g$ is $\gamma$-H\"older continuous along $\mathcal{X}$ if there exists a constant $[g]_{\gamma;\mathcal{X}} >0$ such that \begin{equation*}
        |g(x) - g(y)|_\mathcal{Y} \le [g]_{\gamma;\mathcal{X}} |x-y|^\gamma_\mathcal{X} \qquad \text{for any $x,y \in E$ such that $x-y \in \mathcal{X}$} .
    \end{equation*}
    If $\gamma=1$, we say that $g$ is Lipschitz continuous along $\mathcal{X}$ and we write $|g|_{Lip;\mathcal{X}}$ instead of $[g]_{1;\mathcal{X}}$.
    In the case $E = \mathcal{X}$, we simply get the classical notion of H\"older/Lipschitz continuity, and we simply write $[g]_{\gamma}$ instead of $[g]_{\gamma;E}$. 
\end{defn}

\begin{defn} \label{def:FrechetdifferentiabilityalongsubBanachspaces}
    We say that $g$ is Fréchet differentiable along $\mathcal{X}$ if, for any $x \in E$, the function $\mathcal{X} \ni \xi \mapsto g_x(\xi) := g(x+\xi) \in \mathcal{Y}$ is Fréchet differentiable. We denote by \begin{equation*}
        Dg(x) := D g_x (\xi) \big|_{\xi = 0} \in \mathscr{L}(\mathcal{X},\mathcal{Y})  \quad \text{for any $x \in E$} .
    \end{equation*}
    We say that $g$ is continuously Fréchet differentiable along $\mathcal{X}$ if, in addition, the map $E \ni x \mapsto Dg (x) \in \mathscr{L}(\mathcal{X},\mathcal{Y})$ is continuous. In this case we write $g \in C^1_{\mathcal{X}}(E;\mathcal{Y})$. 
    When $E = \mathcal{X}$, we simply write $C^1(E;\mathcal{Y})$ instead of $C^1_E(E;\mathcal{Y})$, and we get the classical notion of Fréchet differentiability.  
\end{defn}

\begin{defn}
    Let $k \in (1,+\infty)$, and (uniquely) decompose it as $k = N + \gamma$ for some $N \in \mathbb{N}_{\ge 0}$ and $\gamma \in (0,1]$. We say that $g$ is $k$-times continuously Fréchet differentiable along $\mathcal{X}$ - and we write $g \in C^k_\mathcal{X}(E;\mathcal{Y})$ - if, for any $x \in E$, the function $\mathcal{X} \ni \xi \mapsto g_x(\xi) := g(x+\xi) \in \mathcal{Y}$ is $N$-times Fréchet differentiable and, denoting by $D^N g$ its $N$-th Fréchet derivative, the map $E \ni x \mapsto D^N g (x) := D^N g_x (\xi) |_{\xi = 0} \in \mathscr{L}^N(\mathcal{X},\mathcal{Y})$ is $\gamma$-H\"older continuous along $\mathcal{X}$.  
    If, in addition, \begin{equation*}
        |g|_{C^k_b;\mathcal{X}} := |g|_\infty + \sum_{j=1}^N |D^j g|_\infty + [D^N g]_{\gamma;\mathcal{X}} < +\infty
    \end{equation*}
    we write $g \in C^k_{b,\mathcal{X}}(E;\mathcal{Y})$. When $E=\mathcal{X}$, we simply write $C^k_b(E;\mathcal{Y})$ instead of $C^k_{b,E}(E;\mathcal{Y})$. 
\end{defn}

\noindent As a trivial example, notice that when we write $g \in C^2(\R^d;\R)$ we mean that $g$ is continuously Fréchet differentiable and its first Fréchet derivative $Dg$ is Lipschitz continuous from $\R^d$ to $\mathscr{L}^2(\R^d,\R)$. 

\begin{rmk} \label{rmk:counterexampleCD}
    Differentiability along sub-Banach spaces naturally arises as a tool to deal with continuity of higher order Fréchet derivatives. 
    In \cite[Remark 5.80]{CARMONADELARUE_volumeI} the authors exhibit the explicit construction of a smooth and compactly supported function $h:\R^d \to \R$ for which the map $g : L_2(\Omega;\R^d) \to \R$ defined by $g(X) := \E(h(X))$ is not twice continuously Fréchet differentiable. 
    The issue clearly lies in the continuity of the second Fréchet derivative as a function from $L_2(\Omega;\R^d)$ to $\mathscr{L}^2(L_2(\Omega;\R^d),\R)$.
    An easy application of H\"older's inequality shows that $D^2g$ is $((p-2) \wedge 1)$-H\"older continuous along $L_p(\Omega;\R^d)$ whenever $p \in (2,\infty)$, from which it follows that $g \in C^{p\wedge 3}_{b,L_p(\Omega;\R^d)}(L_2(\Omega;\R^d);\R)$ for any $p \in (2,\infty)$. 
    Indeed, assuming $d=1$ for simplicity, for any $Y_1,Y_2 \in L_p(\Omega;\R)$ and for any $X,\bar X \in L_2(\Omega;\R)$ such that $X - \bar X \in L_p(\Omega;\R)$, we have \begin{align*}
        |D^2g(X) [Y_1,Y_2] - D^2g(\bar X) [Y_1,Y_2] | &= |\E( (h''(X) - h''(\bar X)) Y_1 Y_2)| \\
        &\le \|h''(X) -h''(\bar X)\|_r \|Y_1\|_p \|Y_2\|_p
    \end{align*}
    where $\frac{1}{r} + \frac{2}{p} = 1$.
    In the case $r \ge p$ (or, equivalently, $p \in (2,3]$), using boundedness and Lipschitz continuity of $h''$, we get
    \begin{align*} 
        \|h''(X) - h''(\bar X)\|_r &= \E(|h''(X) - h''(\bar X)|^{r-p} |h''(X) - h''(\bar X)|^{p})^\frac{1}{r} \\
        &\le 2^{\frac{r-p}{r}} |h''|_\infty^{\frac{r-p}{r}} |h''|_{Lip}^\frac{p}{r} \|X- \bar X\|_p ^{\frac{p}{r}} ,
    \end{align*}
    where $\frac{r-p}{r}=3-p$ and $\frac{p}{r} = p-2$. {If instead $r < p$} (or, equivalently, $p > 3$) we deduce by H\"older's inequality that \begin{equation*}
        \|h''(X) - h''(\bar X)\|_r \le \|h''(X) - h''(\bar X)\|_p \le |h''|_{Lip} \|X- \bar X\|_p .
    \end{equation*}
\end{rmk}

\section{McKean--Vlasov rough SDEs} \label{section:McKean--ValsovroughSDEs}
Let $T>0$ and let $d,m,n \in \mathbb{N}_{\ge 1}$. Let $\mathbf{\Omega}=(\Omega,\mathcal{F},\{\mathcal{F}_t\}_{t \in [0,T]},\mathbb{P})$ be a complete Polish probability space, equipped with a filtration $\{\mathcal{F}_t\}_{t \in [0,T]}$ such that $\mathcal{F}_0$ contains all the $\mathbb{P}$-null sets. 

\begin{defn} \label{def:stochasticcontrolledroughpath}
    (see \cite[Definition 3.1]{FHL21})
    Let $\alpha \in (0,1]$ and let $W:[0,T] \to \R^n$ be $\alpha$-H\"older continuous. Let $p \in [2,\infty)$ and $q \in [p,\infty]$ and let $V$ be a finite dimensional Banach space. A pair $(X,X')$ is said to be a \textit{stochastic controlled rough path} of $(p,q)$-integrability with values in $V$ if 
\begin{enumerate} \renewcommand{\labelenumi}{\roman{enumi})}
    \item $X:\Omega \times [0,T] \to V$ is progressively measurable and 
    \begin{equation*}
        \|\delta X\|_{\alpha;p,q} := \sup_{0\le s < t\le T} \frac{\|\E(|\delta X_{s,t}|^p|\mathcal{F}_s)^\frac{1}{p}\|_q}{|t-s|^\alpha} < +\infty;
    \end{equation*} 
    \item $X':\Omega \times [0,T] \to \mathscr{L}(\R^n,V)$ is progressively measurable such that $\sup_{t \in [0,T]} \|X'_t\|_q < +\infty$ and
    \begin{equation*}
        \|\delta X'\|_{\alpha;p,q} := \sup_{0\le s < t\le T} \frac{\|\E(|\delta X'_{s,t}|^p|\mathcal{F}_s)^\frac{1}{p}\|_q}{|t-s|^{\alpha}} < +\infty;
    \end{equation*}
    \item denoting by $R^X_{s,t}=\delta X_{s,t} - X'_s \delta W_{s,t} $ for $0 \le s \le t \le T$, it holds that 
    \begin{equation*}
        \|\E_\cdot R^X\|_{2\alpha;q}:=\sup_{0\le s < t \le T} \frac{\|\E(R^X_{s,t} \mid \mathcal{F}_s)\|_q}{|t-s|^{2\alpha}} < + \infty.
    \end{equation*} 
\end{enumerate}
We write $(X,X') \in \mathbf{D}_W^{2\alpha} L_{p,q}([0,T],\mathbf{\Omega};V)$ or simply $(X,X') \in \mathbf{D}_W^{2\alpha} L_{p,q}(V)$, if the time interval and the stochastic basis are clear from the context.
In the case $p=q$, we simply write $\mathbf{D}_W^{2\alpha} L_{p}(V)$ and $\|\cdot\|_{\alpha;p}$ instead of $\mathbf{D}_W^{2\alpha} L_{p,p}(V)$ and $\|\cdot\|_{\alpha;p,p}$, respectively.
\end{defn}

\begin{lemma}
    \emph{(Rough stochastic integration, see \cite[Section 3.2]{FHL21})}
    Let $\alpha \in (\frac{1}{3},\frac{1}{2}]$, let $\mathbf{W}=(W,\mathbb{W}) \in \mathscr{C}^\alpha([0,T];\R^n)$ be a rough path and let $(X,X') \in \mathbf{D}_W^{2\alpha} L_{p,q}([0,T],\mathbf{\Omega};\mathscr{L}(\R^n,V))$ be a stochastic controlled rough path.
    Then there exists a unique continuous and adapted stochastic process $$\left(\int_0^t (X_r,X'_r) \, d\mathbf{W}_r \right)_{t \in [0,T]}$$ on $\mathbf{\Omega}$ taking values in $V$ such that, for any $0 \le s \le t \le T$, \begin{equation*} 
        \Big\| \E\Big(\Big|\int_s^t (X_r,X'_r) \, d\mathbf{W}_r - X_s \delta W_{s,t}  \Big|^p \mid \mathcal{F}_s \Big)^\frac{1}{p} \Big\|_q \lesssim |t-s|^{2\alpha}
    \end{equation*} and \begin{equation*} 
        \Big\| \E \Big( \int_s^t (X_r,X'_r) \, d\mathbf{W}_r - X_s \delta W_{s,t} - X'_s \mathbb{W}_{s,t} \mid \mathcal{F}_s \Big) \Big\|_q \lesssim |t-s|^{3\alpha} .
    \end{equation*}
    Here $\int_s^t (X_r,X'_r) \, d\mathbf{W}_r = \int_0^t (X_r,X'_r) \, d\mathbf{W}_r - \int_0^s (X_r,X'_r) \, d\mathbf{W}_r$. If clear from the context we simply denote $\int_0^t (X_r,X'_r) \, d\mathbf{W}_r$ by $\int_0^t X_r \, d\mathbf{W}_r$. 
    Moreover, for any $t \in [0,T]$, \begin{equation*}
        \int_0^t (X_r,X'_r) \, d\mathbf{W}_r = \lim_{|\pi| \to 0} \sum_{[u,v] \in \pi} X_u \delta W_{u,v} + X'_u \mathbb{W}_{u,v} \qquad \text{in $L^p(\Omega;V)$},
    \end{equation*}
    where the limit is taken along any sequence of partitions of $[0,T]$ whose {mesh size} tends to zero, and \begin{equation*}
        \left( \int_0^\cdot (X_r,X'_r) \, d\mathbf{W}_r , X \right) \in \mathbf{D}_W^{2\alpha} L_{p,q}([0,T],\mathbf{\Omega};V)
    \end{equation*}
    provided that $\sup_{t \in [0,T]} \|X_t\|_q < +\infty$. 
\end{lemma}

\begin{defn} \label{def:stochasticcontrolledvectorfield}
    (generalisation of \cite[Definition 3.4]{FHL25} \footnote{Note that $(f,f') \in \mathscr{D}_W^{2\alpha}C^\kappa_{b,\mathcal{X}}$ in the sense of \cref{def:stochasticcontrolledvectorfield} for some $\kappa \in (\frac{1}{\alpha},3]$ if and only if $(f,f') \in \mathscr{D}_W^{2\alpha}C^\kappa_{b,\mathcal{X}}$ and $(D_xf, D_x f') \in \mathscr{D}_W^{2\alpha}C^{\kappa-1}_{b,\mathcal{X}}$ in the sense of \cite[Definition 3.4]{FHL25}. When $E = \mathcal{X}$, \cref{def:stochasticcontrolledvectorfield} aligns with \cite[Definition 2.3]{BFS24}.})
    Let $E$ be a topological vector space and let $\mathcal{X}, \mathcal{Y}$ be two Banach spaces such that {$\mathcal{X}$ is continuously embedded into $E$}. 
    Let $\alpha \in (0,1]$ and let $W:[0,T] \to \R^n$ be $\alpha$-H\"older continuous. 
    Let $\kappa \in (1,\infty)$ such that $\kappa = N + \gamma$ for some $N \in \mathbb{N}_{\ge 1}$ and $\gamma \in (0,1]$. 
    A \textit{(deterministic) controlled vector field} in $\mathscr{D}_W^{2\alpha}C_{b,\mathcal{X}}^\kappa(E;\mathcal{Y})$
    - or simply in $\mathscr{D}_W^{2\alpha}C_{b,\mathcal{X}}^\kappa$
    - is a pairing $(f,f')$ satisfying the following: \begin{enumerate}  \renewcommand{\labelenumi}{\roman{enumi})}
        \item $f: [0,T] \to C^\kappa_{b,\mathcal{X}}(E;\mathcal{Y})$ and $f': [0,T] \to C^{\kappa-1}_{b,\mathcal{X}}(E;\mathscr{L}(\R^n,\mathcal{Y}))$ are 
        such that
        \begin{equation*}
            \sup_{t \in [0,T]}  |f_t(\cdot)|_{C_b^\kappa;\mathcal{X}}  + \sup_{t \in [0,T]}  |f'_t(\cdot)|_{C_b^{\kappa-1};\mathcal{X}}  < +\infty;
        \end{equation*}
        \item denoting by $\delta f_{s,t}(x) := f_t(x) - f_s(x)$,  \begin{equation*}
            \sup_{0 \le s < t \le T} \frac{ \sup_{x \in E}|\delta f_{s,t} (x)|}{|t-s|^\alpha} < +\infty
        \end{equation*}
        and the same holds by replacing $\delta f_{s,t}$ by {$\delta (D f)_{s,t}, \dots , \delta (D^N f)_{s,t}$ and $ \delta f'_{s,t}, \dots, \delta (D^{N-1} f')_{s,t}$};
        \item denoting by $R^f_{s,t}(x) := f_t(x) - f_s(x) - f_s'(x) \delta W_{s,t}$,
        \begin{equation*}
            \sup_{0 \le s < t \le T} \frac{ \sup_{x \in E}|R^f_{s,t} (x) |}{|t-s|^{2\alpha}} < +\infty
        \end{equation*}
        and the same holds replacing $R^f_{s,t}$ by $R^{Df}_{s,t}, \dots, R^{D^{N-1}f}_{s,t}$.
    \end{enumerate}
\end{defn}

\noindent Consider the following measurable mappings: \begin{align*}
    b &:  [0,T] \times \R^d \times \mathcal{P}_2(\R^d) \to \R^d \\
    \sigma =(\sigma_l)_{l=1,\dots,m} &:   [0,T] \times \R^d \times \mathcal{P}_2(\R^d) \to \mathscr{L}(\R^m,\R^d) \equiv \R^{d \times m} \\
    f = (f_\kappa)_{\kappa=1,\dots,n} &:   [0,T] \times \R^d \times \mathcal{P}_2(\R^d) \to \mathscr{L}(\R^n,\R^d) \\
    f' = (f'_{\kappa \lambda})_{\kappa,\lambda = 1,\dots,n} &:   [0,T] \times \R^d \times \mathcal{P}_2(\R^d) \to \mathscr{L}(\R^n \otimes \R^n,\R^d) .
\end{align*}


{
Let $(\Omega',\mathcal{F}',\mathbb{P}')$ be  an auxiliary, atomless Polish probability space, introduced to lift the coefficients of the McKean--Vlasov RSDE in the sense of Lions.
The following provides a way to transfer random variables onto this auxiliary space. 
}
\begin{fact} \phantomsection\label{fact}
    (see \cite[Appendix A.1]{FHL25})
    There is a measurable and measure preserving map $\pi : \Omega' \to \Omega$ such that, for any separable Banach space $E$ and for any $E$-valued random variable $X$ on $(\Omega,\mathcal{F},\mathbb{P})$, the map \begin{equation*}
        \slashed{X}(\omega') := X(\pi(\omega')), \qquad \omega' \in \Omega'
    \end{equation*}
    defines a $\slashed{\mathcal{F}}:= \pi^{-1}(\mathcal{F})$-measurable random variable whose law is identical to the law of $X$, which also implies the isometry property of the map $L_p(\Omega,\mathcal{F};E) \ni X \mapsto \slashed{X} \in L_p(\Omega',\slashed{\mathcal{F}};E)$.
    More generally, the construction preserves joint laws. 
\end{fact}

\noindent Let $\{\mathcal{F}'_t\}_{t \in [0,T]}$ be a filtration on $(\Omega',\mathcal{F}',\mathbb{P}')$ and assume that $\slashed{\mathcal{F}} \subseteq \mathcal{F}'$ and $\slashed{\mathcal{F}}_t := \pi^{-1}(\mathcal{F}_t) \subseteq \mathcal{F}'_t$ for any $t \in [0,T]$. 
Let $\hat{b}: [0,T] \times \R^d \times L_2(\Omega';\R^d) \to \R^d$ be the Lions lift of $b$ with respect to $(\Omega',\mathcal{F}',\mathbb{P}')$. Namely, $\hat{b}_t(x,X) := b(t, x,\mathcal{L}_X)$. In a similar way, we define $\hat{\sigma}, \hat{f}$ and $\hat{f'}$.

Let $B=(B_t)_{t \in [0,T]}$ be an $m$-dimensional $\{\mathcal{F}_t\}_t$-Brownian motion on $\mathbf{\Omega}$, and let $\mathbf{W}=(W,\mathbb{W}) \in \mathscr{C}^\alpha([0,T];\R^n)$ be a deterministic rough path with $\alpha \in (\frac{1}{3},\frac{1}{2}]$. 
{Let $p \in [2,\infty)$ and assume that the map $(x,X) \mapsto \hat f_t(x,X)$ is continuously Fréchet differentiable along $\R^d \times L_p(\Omega';\R^d)$ in the sense of \cref{def:FrechetdifferentiabilityalongsubBanachspaces}, for any $t \in [0,T]$.  
Then its Fréchet derivative $D \hat f_t(x,X)$ is a bounded linear operator 
\begin{equation*}
    D \hat f_t(x,X)[y,Y] = D_1 \hat f_t(x,X)[y] + D_2 \hat f_t(x,X)[Y], \qquad (y,Y) \in \R^d \times L_p(\Omega';\R^d), 
\end{equation*}
where $D_1$ denotes the classical (partial) Fréchet derivative with respect to $x \in \R^d$ and $D_2$ indicates differentiation with respect to $X \in L_2(\Omega';\R^d)$ along $L_p(\Omega';\R^d)$. 
}

\begin{defn} \label{def:solution_McKeanVlasovroughSDE}
    (see \cite[Definition 3.7]{FHL25})
    An $L_{p,\infty}$-integrable solution to \begin{equation} \label{eq:solution_McKeanVlasovroughSDE}
    \begin{cases}
        dX_t = b(t, X_t,\mu_t) \, dt + \sigma(t, X_t,\mu_t) \, dB_t + (f,f')(t, X_t,\mu_t) \, d\mathbf{W}_t \\
        \mu_t = \mathcal{L}_{X_t}
    \end{cases}
    \end{equation}
    is a continuous and adapted stochastic process $X$ on $\mathbf{\Omega}$ such that the following conditions are satisfied: \begin{enumerate} \renewcommand{\labelenumi}{(\roman{enumi})}
        \item for any $t \in [0,T]$, $\int_0^t |\hat{b}_r(X_r,\slashed{X}_r)| dr$ and $\int_0^t |(\hat{\sigma} \hat{\sigma}^\top)_r(X_r,\slashed{X}_r)| dr$ are finite $\mathbb{P}$-a.s. ;
        \item $(\hat{f}(X,\slashed{X}), D\hat{f}(X,\slashed{X})[\hat{f}(X,\slashed{X}),\cancel{\hat{f}(X,\slashed{X})}]  + \hat{f}'(X,\slashed{X}))$ is a stochastic controlled rough path in $\mathbf{D}_W^{2\alpha}L_{p,\infty}([0,T],\mathbf{\Omega};\mathscr{L}(\R^n,\R^d))$ ;
        \item for any $0 \le s \le t \le T$, \begin{multline} \label{eq:Davieexpansion_solutionMcKeanVlasovroughSDE}
            \delta X_{s,t} = \int_s^t \hat{b}_r(X_r,\slashed{X}_r) \, dr + \sum_{l=1}^m \int_s^t (\hat{\sigma}_l)_r (X_r,\slashed{X}_r) \, dB^l_r + \sum_{\kappa=1}^n (\hat{f}_\kappa)_s(X_s,\slashed{X}_s) \,  \delta W^\kappa_{s,t}  \\ + \sum_{\kappa,\lambda=1}^n \big(D(\hat{f}_\lambda)_s(X_s,\slashed{X}_s) [(\hat{f}_\kappa)_s(X_s,\slashed{X}_s), \cancel{(\hat{f}_\kappa)_s(X_s,\slashed{X}_s)}] + (\hat{f}'_{\lambda \kappa})_s(X_s,\slashed{X}_s) \big) \, \mathbb{W}^{\kappa \lambda}_{s,t} + X^\natural_{s,t} ,
            \end{multline}
            where $\| \E( |X^\natural_{s,t}|^p  \mid \mathcal{F}_s)^\frac{1}{p} \|_\infty \lesssim |t-s|^{2\alpha}$ and $\|\E(X^\natural_{s,t} \mid \mathcal{F}_s)\|_\infty\lesssim |t-s|^{3\alpha}$ .
    \end{enumerate}
    If the $\mathcal{F}_0$-measurable initial condition $X_0=\xi$ is specified, we say that $X$ is a solution to \eqref{eq:solution_McKeanVlasovroughSDE} starting from $\xi$.
\end{defn}

\begin{rmk} \label{rmk:FréchetderivativesinDavieexpansion}
    Assume that, for any $(t,\mu) \in [0,T] \times \mathcal{P}_2(\R^d)$, the function $\R^d \ni x \mapsto f(t, x,\mu)$ is continuously differentiable in classical sense and that, for any $(t,x) \in [0,T] \times \R^d$, the function $\mathcal{P}_2(\R^d) \ni \mu \mapsto f(t, x,\mu)$ is continuously L-differentiable. Then, necessarily, $D_2 \hat f_t(x,X)[Y] = \E(\partial_\mu f(t,x,\mu) (X) \cdot Y)$ for any $Y \in L_p(\Omega';\R^d)$.
    Hence, for any $t \in [0,T]$ and for any $\kappa,\lambda =1,\dots,n$, we can rewrite  \begin{align*} 
        &D(\hat{f}_\lambda)_t(X_t,\slashed{X}_t) [(\hat{f}_\kappa)_t(X_t,\slashed{X}_t), \cancel{(\hat{f}_\kappa)_t(X_t,\slashed{X}_t)}]   \\
        &= D_1(\hat{f}_\lambda)_t(X_t,\slashed{X}_t) [(\hat{f}_\kappa)_t(X_t,\slashed{X}_t)]  +D_2(\hat{f}_\lambda)_t(X_t,\slashed{X}_t) [\cancel{(\hat{f}_\kappa)_t(X_t,\slashed{X}_t)}]   \\ 
        &= \sum_{j=1}^d \partial_{x^j} f_\lambda(t, X_t, \mathcal{L}_{X_t}) f_\kappa^j(t, X_t, \mathcal{L}_{X_t}) + \int_{\R^d} \partial_\mu f_\lambda (t, X_t,\mathcal{L}_{X_t}) (z) \cdot f_\kappa(t, z,\mathcal{L}_{X_t}) \, \mathcal{L}_{X_t}(dz) ,
    \end{align*}
    where $\cdot$ denotes the scalar product in $\R^d$ and we used the fact that $\mathcal{L}_{\slashed{X}_t} = \mathcal{L}_{X_t}$. 
    Notice that, $\cancel{(\hat{f}_\kappa)_t(X_t,\slashed{X}_t)}(\omega') = (\hat{f}_\kappa)_t(X_t(\omega'),\slashed{X}_t)$. 
    Hence, by using the definitions of $\hat{b}$, $\hat{\sigma}$ and the previous computation, 
    \eqref{eq:Davieexpansion_solutionMcKeanVlasovroughSDE} can be written as \begin{multline*} 
            \delta X_{s,t} = \int_s^t b(s, X_r,\mathcal{L}_{X_r}) \, dr + \sum_{l=1}^m \int_s^t \sigma_l (r, X_r,\mathcal{L}_{X_r}) \, dB_r + \sum_{\kappa=1}^n f_\kappa(s, X_s,\mathcal{L}_{X_s}) \,  \delta W^\kappa_{s,t}  \\ + \sum_{\kappa,\lambda=1}^n \Big(\sum_{j=1}^d \partial_{x^j} f_\lambda(s, X_s, \mathcal{L}_{X_s}) f_\kappa^j(s, X_s, \mathcal{L}_{X_s}) + \int_{\R^d} \partial_\mu f_\lambda (s, X_s,\mathcal{L}_{X_s}) (z) \cdot f_\kappa(s, z,\mathcal{L}_{X_s}) \, \mathcal{L}_{X_s}(dz)  \\ + f'_{\lambda \kappa}(s, X_s,\mathcal{L}_{X_s}) \Big) \, \mathbb{W}^{\kappa \lambda}_{s,t} + X^\natural_{s,t} .
            \end{multline*}
\end{rmk}

\begin{thm} \label{thm:wellposedness_McKeanVlasovroughSDEs}
    (see \cite[Theorem 3.10]{FHL25})
    Let $\gamma \in (\frac{1}{\alpha},3]$, let $p \in [2,\infty)$. 
    Assume that, for any $t \in [0,T]$,  $$\hat b_t \in C^1_{b, \R^d \times L_p(\Omega';\R^d)} (\R^d \times L_2(\Omega';\R^d) ; \R^d)$$ with $\sup_{t \in [0,T]} (|\hat b(t,\cdot , \cdot)|_{\infty} + [\hat b(t,\cdot , \cdot)]_{1;\R^d \times L_p(\Omega;\R^d)}) < +\infty$.  
    Assume similar properties on $\hat \sigma_t$ and suppose that
    \begin{equation*}
        (\hat{f},\hat{f}') \in \mathscr{D}_W^{2\alpha} C_{b, \R^d \times L_p(\Omega';\R^d)}^\gamma(\R^d \times L_2(\Omega';\R^d); \mathscr{L}(\R^n,\R^d)) .
    \end{equation*}  
    Then, for any $\xi \in L_2(\Omega, \mathcal{F}_0;\R^d)$, there is a unique $L_{p,\infty}$-integrable solution $X$ to \eqref{eq:solution_McKeanVlasovroughSDE} starting from $\xi$. In particular, $(X,\hat{f}(X,\slashed{X})) \in \mathbf{D}_W^{2\alpha}L_{p,\infty}([0,T], \mathbf{\Omega}; \R^d)$.
    
\end{thm}

\section{Nonlinear rough PDEs} \label{section:solvingmeanfieldroughPDEs}
Let $T>0$ and let $d,m,n \in \mathbb{N}_{\ge 1}$. 
We consider the mappings \begin{equation} \label{eq:coefficients_meanfieldRPDE}
    \begin{aligned}
        b&
    : [0,T] \times \R^d \times \mathcal{P}_2(\R^d) \to \R^d \\
    \sigma_l &
    : [0,T] \times \R^d \times \mathcal{P}_2(\R^d) \to \R^d \qquad l=1,\dots,m \\
    f_\kappa &
    : [0,T] \times \R^d \times \mathcal{P}_2(\R^d) \to \R^d \qquad \kappa=1,\dots,n.
    \end{aligned}
\end{equation}
We define the map $a:[0,T] \times \R^d \times \mathcal{P}_2(\R^d) \to \R^{d\times d}$
via $$a^{ij}(t,x,\mu) := \sum_{l=1}^m \sigma^i_l(t,x,\mu) \sigma^j_l(t,x,\mu) \qquad i,j=1,\dots,d .$$
Notice that the matrix $a(t,x,\mu)$ is by construction symmetric and nonnegative definite. 
For any $(t,\mu) \in [0,T] \times \mathcal{P}_2(\R^d)$, the (respectively, second and first order) time-dependent differential operators $L_t[\mu]$ and $\Gamma_t[\mu]=(\Gamma_t[\mu]_1,\dots, \Gamma_t[\mu]_n)$ are defined, for any suitably regular $\varphi:\R^d \to \R$ and for any $x \in \R^d$, as 
\begin{align*}
    L_t[\mu] \varphi (x) :&= \frac{1}{2} \partial^2_{ij} \varphi(x) a^{ij}(t,x,\mu) + \partial_i \varphi (x) b^i(t,x,\mu) \\ 
    \Gamma_t[\mu]_\kappa \varphi(x) & := \partial_i \varphi(x) f^i_\kappa(t,x,\mu) \qquad \kappa=1,\dots,n , 
\end{align*}
with Einstein's repeated indices summation convention and being $\partial_i$ the partial derivative with respect to the space variable $x^i$.
Their formal adjoint operators are respectively
\begin{align*}
    (L_t[\mu])^\star \varphi(x) &= \frac{1}{2} \partial^2_{ij} (a^{ij}(t,\cdot,\mu) \varphi)(x)  - \partial_i (b^i(t,\cdot,\mu) \varphi)(x) \\ 
    (\Gamma_t[\mu]_\kappa)^\star \varphi(x) & = - \partial_i (f_\kappa^i(t,\cdot,\mu) \varphi) (x) \qquad \kappa=1,\dots,n . 
\end{align*}
Later in the paper we assume that
the map $\mu \mapsto f_\kappa(t,x,\mu)$ is component-by-component continuously L-differentiable, denoting by $\partial_\mu f_\kappa(t,x,\mu)(\cdot)$ its L-derivative, and that $f_\kappa(t,x,\mu) - f_\kappa(s,x,\mu) \approx f'_{\kappa \lambda} (s,x,\mu) \delta Y^\lambda_{s,t}$ for some functions 
\begin{equation*}
    f'_{\kappa \lambda} 
    : [0,T] \times \R^d \times \mathcal{P}_2(\R^d) \to \R^d \qquad \kappa,\lambda=1,\dots,n  .
\end{equation*} 
This justifies the definition of the first order differential operator $$ \Gamma'_t[\mu] = \begin{pmatrix}
\Gamma'_t[\mu]_{11} & \cdots & \Gamma'_t[\mu]_{1n} \\
\vdots & \ddots & \vdots \\
\Gamma'_t[\mu]_{n1} &  \cdots & \Gamma'_t[\mu]_{nn}
\end{pmatrix}$$  as 
\begin{equation*}
    \Gamma_t'[\mu]_{\kappa\lambda} \varphi(x) := \partial_i \varphi(x) \int_{\R^d} \partial_\mu f^i_\kappa (t,x,\mu)(z) \cdot f_\lambda(t,z,\mu) \, \mu(dz)  + \partial_i \varphi(x) (f'_{\kappa\lambda})^i(t,x,\mu) ,
\end{equation*}
where $\cdot$ denotes the scalar product in $\R^d$. 
The mappings $\sigma:[0,T] \times \R^d \times \mathcal{P}_2(\R^d) \to \mathscr{L}(\R^m,\R^d)$ and $f:[0,T] \times \R^d \times \mathcal{P}_2(\R^d) \to \mathscr{L}(\R^n,\R^d)$ are defined via $\sigma_t(x,\mu) y := \sigma_l(t,x,\mu) y^l$ for any $y \in \R^m$, and $f_t(x,\mu) z := f_\kappa(t,x,\mu) z^\kappa$ for any $z \in \R^n$, respectively. In a similar way $f':[0,T] \times \R^d \times \mathcal{P}_2(\R^d) \to \mathscr{L}(\R^n \otimes \R^n ,\R^d)$ denotes the map $f_t'(x,\mu) w := f'_{\kappa\lambda}(t,x,\mu) w^{\kappa \lambda}$ for any $w \in \R^n \otimes \R^n$. 

\begin{notation}
    For any finite measure $\mu$ on $\R^d$ and for any bounded Borel measurable function $\varphi:\R^d \to \R$, we denote by \begin{equation*}
    \langle \mu , \varphi \rangle := \int_{\R^d} \varphi (x) \, \mu(dx) . 
\end{equation*}
\end{notation}

\noindent Let $\alpha \in (\frac{1}{3},\frac{1}{2}]$ and let $\mathbf{W}=(W,\mathbb{W}) \in \mathscr{C}_g^\alpha([0,T];\R^n)$ be a weakly geometric rough path.
 
\begin{defn} \label{def:solution_meanfieldroughPDE}
    A solution to \footnote{In the literature, one can also find equation \eqref{eq:meanfieldroughPDEs} written as \begin{equation*}
    d\mu_t = \left(\frac{1}{2} \nabla^2_x : (a(t,\cdot,\mu_t) \mu_t) - \nabla_x \cdot (b(t,\cdot,\mu_t) \mu_t)  \right)  dt - \nabla_x \cdot (f(t,\cdot,\mu_t) \mu_t) \, d\mathbf{W}_t 
\end{equation*} or \begin{equation*}
    d\mu_t = (L_t[\mu_t])^\star \mu_t \, dt + (\Gamma_t[\mu_t])^\star \mu_t \, d\mathbf{W}_t .
\end{equation*}} \begin{equation} \label{eq:meanfieldroughPDEs}
        d\mu_t = \left(\frac{1}{2} \partial^2_{ij} (a^{ij}(t,\cdot,\mu_t) \mu_t) -\partial_i (b^i(t,\cdot,\mu_t)\mu_t) \right) \, dt -  \partial_i (f^i(t,\cdot,\mu_t) \mu_t) \, d\mathbf{W}_t
    \end{equation}
    is a weakly-continuous map $\mu : [0,T] \to \mathcal{P}_2(\R^d)$ such that, for any test function $\varphi \in C^3_b(\R^d;\R)$, the following conditions are satisfied: 
    \begin{enumerate} \renewcommand{\labelenumi}{(\roman{enumi})}
        \item the map $t \mapsto \langle \mu_t, L_t[\mu_t] \varphi \rangle$ belongs to $L_1([0,T];\R)$ ; \label{prova}
        \item for any $s,t \in [0,T]$, for any $\kappa,\lambda=1,\dots,n$, and uniformly over bounded sets of $\varphi$ in $C^3_b(\R^d;\R)$,
        \begin{align*}
              &\langle \mu_t, (\Gamma_t[\mu_t]_\kappa \Gamma_t[\mu_t]_\lambda + \Gamma'_t[\mu_t]_{\lambda \kappa} ) \varphi  \rangle \overset{\alpha}{=} \langle \mu_s, (\Gamma_s[\mu_s]_\kappa \Gamma_s[\mu_s]_\lambda + \Gamma'_s[\mu_s]_{\lambda \kappa} ) \varphi  \rangle    \\
              &\langle \mu_t, \Gamma_t[\mu_t]_\kappa \varphi  \rangle - \langle \mu_s, \Gamma_s[\mu_s]_\kappa \varphi  \rangle  \overset{2\alpha}{=} \langle \mu_s, (\Gamma_s[\mu_s]_\eta \Gamma_s[\mu_s]_\kappa + \Gamma'_s[\mu_s]_{\kappa\eta} ) \varphi  \rangle \delta W_{s,t}^\eta ;
        \end{align*}
        \item for any $0 \le s \le t \le T$ \begin{multline*}
            \langle \mu_t, \varphi \rangle - \langle \mu_s, \varphi \rangle = \int_s^t \langle \mu_r, L_r[\mu_r] \varphi \rangle \, dr + \langle \mu_s, \Gamma_s[\mu_s]_\kappa \varphi  \rangle \delta W^\kappa_{s,t}  \\ + \langle \mu_s, (\Gamma_s[\mu_s]_\kappa \Gamma_s[\mu_s]_\lambda + \Gamma'_s[\mu_s]_{\lambda \kappa}) \varphi  \rangle \mathbb{W}^{\kappa\lambda}_{s,t} + \mu_{s,t}^{\natural,\varphi},
        \end{multline*}
        with $\mu_{s,t}^{\natural,\varphi}=o(|t-s|)$ as $|t-s| \to 0$. 
    \end{enumerate}
    When the initial condition $\mu_0 = \nu \in \mathcal{P}_2(\R^d)$ is specified, we say that $\mu$ is a solution to \eqref{eq:meanfieldroughPDEs} starting from $\nu$. 
\end{defn}

\begin{prop}
    Let $\mu : [0,T] \to \mathcal{P}_2(\R^d)$ satisfying condition (ii) of \cref{def:solution_meanfieldroughPDE}.
    Then the following limit exists for any $t \in [0,T]$ along any sequence of partitions $\pi$ of $[0,t]$ whose mesh size tends to zero: 
    \begin{multline*}
        \int_0^t  \langle \mu_r, \Gamma_r[\mu_r] \varphi  \rangle d \mathbf{W}_r := \lim_{|\pi|\to0} \sum_{[s,u]\in \pi} \big(  \langle \mu_s, \Gamma_s[\mu_s]_\kappa \varphi  \rangle \delta W^\kappa_{s,u}  \\ + \langle \mu_s, (\Gamma_s[\mu_s]_\kappa \Gamma_s[\mu_s]_\lambda + \Gamma'_s[\mu_s]_{\lambda \kappa}) \varphi  \rangle \mathbb{W}^{\kappa\lambda}_{s,u} \big).
    \end{multline*}
    In particular, condition (iii) of \cref{def:solution_meanfieldroughPDE} can be equivalently replaced by   \begin{itemize}
    \item[(iii')] for any $t \in [0,T]$ \begin{equation*}
        \langle \mu_t, \varphi \rangle = \langle \mu_0, \varphi \rangle + \int_0^t \langle \mu_r, L_r[\mu_r] \varphi \rangle \, dr + \int_0^t  \langle \mu_r, \Gamma_r[\mu_r] \varphi  \rangle d \mathbf{W}_r ,
    \end{equation*} 
\end{itemize}
and the quantity $\mu_{s,t}^{\natural,\varphi}$ in fact satisfies $|\mu_{s,t}^{\natural,\varphi}| \le C_{\varphi} |t-s|^{3\alpha}$, where $C_\varphi > 0$ 
is uniformly over bounded sets of $\varphi$ in $C^3_b(\R^d;\R)$.
\end{prop} 

\begin{proof}
    The proof is essentially a consequence of the sewing lemma (see \cite[Lemma 4.2]{FRIZHAIRER}). 
    Let $t \in [0,T]$. For any $\varphi \in C^3_b(\R^d;\R)$ and for any $s,u \in [0,t]$ define \begin{equation*}
        \Xi_{s,u}^\varphi := \langle \mu_s, \Gamma_s[\mu_s]_\kappa \varphi  \rangle \delta W^\kappa_{s,u}  + \langle \mu_s, (\Gamma_s[\mu_s]_\kappa \Gamma_s[\mu_s]_\lambda + \Gamma'_s[\mu_s]_{\lambda \kappa}) \varphi  \rangle \mathbb{W}^{\kappa\lambda}_{s,u}. 
    \end{equation*}
    Recalling \eqref{eq:chenrelation}, for any $s,u,v \in [0,t]$ we have \begin{align*}
        \Xi_{s,v}^\varphi - \Xi_{s,u}^\varphi - \Xi_{u,v}^\varphi &= (\langle \mu_s, \Gamma_s[\mu_s]_\kappa \varphi  \rangle - \langle \mu_u, \Gamma_u[\mu_u]_\kappa \varphi  \rangle) \delta W^\kappa_{u,v}  \\
        & \quad + \langle (\mu_s, \Gamma_s[\mu_s]_\kappa \Gamma_s[\mu_s]_\lambda + \Gamma'_s[\mu_s]_{\lambda \kappa}) \varphi \rangle (\mathbb{W}^{\kappa\lambda}_{u,v} + \delta W^\kappa_{s,u} \delta{W}^\lambda_{u,v})  \\
        & \quad - \langle \mu_u,  (\Gamma_u[\mu_u]_\kappa \Gamma_u[\mu_u]_\lambda + \Gamma'_u[\mu_u]_{\lambda \kappa})\varphi \rangle \mathbb{W}^{\kappa\lambda}_{u,v}.  
    \end{align*}
        From condition (ii) of \cref{def:solution_meanfieldroughPDE} we deduce that $|\Xi_{s,v}^\varphi - \Xi_{s,u}^\varphi - \Xi_{u,v}^\varphi| \le C_\varphi |t-s|^{3\alpha}$ and $C_\varphi$ is uniform over bounded sets of $\varphi \in C^3_b(\R^d;\R)$. 
\end{proof}

 Our main result is the following

\begin{thm} \label{thm:mainresult}
    Under \cref{assumptions_existence} and \cref{assumptions_uniqueness} below and for any $\nu \in \mathcal{P}_2(\R^d)$, there exists a unique solution to \eqref{eq:meanfieldroughPDEs} starting from $\nu$ {and satisfying conditions 1.\ and 2.\ of \cref{thm:uniqueness_meanfieldroughPDEs}}. 
    Moreover, such a solution can be characterized as the curve of one-dimensional marginals of the $L_{4,\infty}$-integrable solution to a McKean--Vlasov rough SDE of the form \eqref{eq:McKean-VlasovroughSDE_existence} below. 
\end{thm}

\begin{proof}
    The proof follows by combining \cref{thm:existence_meanfieldroughPDE} with \cref{thm:uniqueness_meanfieldroughPDEs} and \cref{prop:stabilityundercomposition_uniqueness}. 
\end{proof}

\subsection{Existence} \label{section:existence}
{
We prove existence for \eqref{eq:meanfieldroughPDEs} by showing that the law of the associated McKean--Vlasov rough SDE solves the equation. 
We therefore need to impose sufficient conditions on the coefficients $b,\sigma,f,f'$ so that their Lions lifts $\hat b, \hat \sigma, \hat f, \hat f'$ satisfy the well-posedness assumptions of \cref{thm:wellposedness_McKeanVlasovroughSDEs}. 
In \cref{assumptions_existence} below, we essentially require $b$ and $\sigma$ to be Lipschitz continuous in $(x,\mu)$, while $f$ should be twice continuously differentiable in $(x,\mu)$ with Lipschitz continuous second derivatives. Derivatives with respect to the variable $\mu$ are understood in the Lions sense. 
Since we are in a rough path framework,  time dependence requires extra care. That is, $f$ should also be controlled in time by $f'$. 
}

\begin{assumptions} \label{assumptions_existence}
Let $b:[0,T] \times \R^d \times \mathcal{P}_2(\R^d) \to \R^d$ and $\sigma:[0,T] \times \R^d \times \mathcal{P}_2(\R^d) \to \mathscr{L}(\R^m,\R^d)$ be Borel measurable, globally bounded and Lipschitz continuous in $(x,\mu)$ uniformly in $t$.
Let $f:[0,T] \times \R^d \times \mathcal{P}_2(\R^d) \to \mathscr{L}(\R^n,\R^d)$ be Borel measurable and globally bounded, and suppose that  \begin{enumerate} \renewcommand{\labelenumi}{(\alph{enumi})}
        
        \item for any $(t, \mu) \in [0,T] \times \mathcal{P}_2 (\mathbb{R}^d)$, the map $\mathbb{R}^d \ni x \mapsto f (t, x, \mu)$ is twice differentiable in classical sense. Moreover, its first derivative $D_x f  : [0, T] \times \mathbb{R}^d \times \mathcal{P}_2    (\mathbb{R}^d) \rightarrow \mathscr{L} (\mathbb{R}^n , \mathbb{R}^d \otimes \mathbb{R}^d)$ is globally bounded, and its second derivative $D^2_xf  : [0, T] \times \mathbb{R}^d \times \mathcal{P}_2 (\mathbb{R}^d) \rightarrow \mathscr{L} (\mathbb{R}^n , \mathbb{R}^d \otimes \mathbb{R}^d \otimes \mathbb{R}^d)$ is globally bounded, Lipschitz continuous in $(x,\mu)$ uniformly in $t$ and $\alpha$-H\"older continuous in $t$ uniformly in $(x, \mu)$;
        
        \item for any $(t, \mu) \in [0,T] \times \mathcal{P}_2 (\mathbb{R}^d)$, the map $\mathcal{P}_2 (\mathbb{R}^d) \ni \mu \mapsto D_x f (t, x,\mu)$ is component-by-component continuously L-differentiable. Moreover, $\partial_{\mu} D_xf : [0, T] \times \mathbb{R}^d \times \mathcal{P}_2 (\mathbb{R}^d) \times \mathbb{R}^d \rightarrow \mathscr{L} (\mathbb{R}^n , \mathbb{R}^d\otimes \mathbb{R}^d \otimes \mathbb{R}^d)$ has a $\mu$-version such that $\partial_{\mu} D_x f (t, x, \mu) (v)$ is globally bounded, Lipschitz continuous in $(x, \mu, v)$ uniformly in $t$ and $\alpha$-H\"older continuous in $t$ uniformly in $(x, \mu, v)$;
        
        \item for any $(t, x) \in [0, T] \times \mathbb{R}^d$, the map $\mathcal{P}_2 (\mathbb{R}^d) \ni \mu \mapsto f (t, x, \mu)$ is component-by-component continuously L-differentiable. Moreover, $\partial_{\mu} f : [0, T] \times \mathbb{R}^d \times \mathcal{P}_2 (\mathbb{R}^d) \times \mathbb{R}^d \rightarrow \mathscr{L} (\mathbb{R}^n , \mathbb{R}^d \otimes \mathbb{R}^d)$ has a $\mu$-version such that $\partial_{\mu} f (t, x, \mu) (v)$ is uniformly bounded in $(t, x, \mu, v)$; 
    
        \item for any $(t, x, \mu) \in [0, T] \times \mathbb{R}^d \times \mathcal{P}_2 (\mathbb{R}^d)$, the version of $\mathbb{R}^d \ni v \mapsto \partial_{\mu} f (t, x, \mu) (v)$ in (c) is differentiable in classical sense. Moreover, $D_v \partial_{\mu} f : [0,T] \times \mathbb{R}^d \times \mathcal{P}_2 (\mathbb{R}^d) \times \mathbb{R}^d \rightarrow \mathscr{L} (\mathbb{R}^n , \mathbb{R}^d \otimes \mathbb{R}^d \otimes \mathbb{R}^d)$ is globally bounded, Lipschitz continuous in $(x, \mu, v)$ uniformly in $t$ and $\alpha$-H\"older continuous in $t$ uniformly in $(x, \mu, v)$;
    
        \item for any $(t, x, v) \in [0, T] \times \mathbb{R}^d \times \mathbb{R}^d$, the version of $\mathcal{P}_2 (\mathbb{R}^d) \ni \mu \mapsto \partial_{\mu} f (t, x, \mu) (v)$ in (c) is component-by-component continuously L-differentiable. Moreover, its Lions derivative $\partial_{\mu}^2 f : [0,T] \times \R^d \times \mathcal{P}_2(\R^d) \times \R^d \times \R^d \to \mathscr{L}(\R^n,\R^d \otimes \R^d \otimes \R^d)$ is globally bounded, Lipschitz continuous in $(x, \mu, v, v')$ uniformly in $t$ and $\alpha$-H\"older continuous in $t$ uniformly in $(x, \mu, v, v')$;
    
        \item for any $(t, \mu, v) \in [0, T] \times \mathcal{P}_2 (\mathbb{R}^d) \times \mathbb{R}^d$, the version of $\mathbb{R}^d \ni x \mapsto \partial_{\mu} f (t, x, \mu) (v)$ in (c) is differentiable in classical sense. Moreover, its derivative $D_x \partial_{\mu} f : [0,T] \times \R^d \times \mathcal{P}_2(\R^d) \times \R^d \to \mathscr{L}(\R^n,\R^d \otimes \R^d \otimes \R^d)$ is globally bounded, Lipschitz continuous in $(x, \mu, v)$ uniformly in $t$ and $\alpha$-H\"older continuous in $t$ uniformly in $(x, \mu, v)$.
    \end{enumerate}
    Let $f':[0,T] \times \R^d \times \mathcal{P}_2(\R^d) \to \mathscr{L}(\R^n \otimes \R^n ,\R^d)$ be Borel measurable and globally bounded, and assume that 
    \begin{enumerate} \renewcommand{\labelenumi}{(\alph{enumi})} \setcounter{enumi}{6}        
        \item for any $(t, \mu) \in [0, T] \times \mathcal{P}_2 (\mathbb{R}^d)$, the map $\mathbb{R}^d \ni x \mapsto f' (t, x, \mu)$ is differentiable in classical sense. Moreover, $D_x f ' : [0, T] \times \mathbb{R}^d \times \mathcal{P}_2 (\mathbb{R}^d) \rightarrow \mathscr{L} (\mathbb{R}^n \otimes \mathbb{R}^n , \mathbb{R}^d \otimes \mathbb{R}^d)$ is globally bounded, Lipschitz continuous in $(x, \mu)$ uniformly in $t$ and $\alpha$-H\"older continuous in $t$ uniformly in $(x, \mu)$;
    
        \item for any $(t, x) \in [0, T] \times \mathbb{R}^d$, the map $\mathcal{P}_2 (\mathbb{R}^d) \ni \mu \mapsto f' (t, x, \mu)$ is component-by-component continuously L-differentiable. Moreover, $\partial_{\mu} f' : [0, T]  \times \mathbb{R}^d \times \mathcal{P}_2 (\mathbb{R}^d) \times \mathbb{R}^d \rightarrow \mathscr{L} (\mathbb{R}^n , \mathbb{R}^d \otimes \mathbb{R}^d)$ has a $\mu$-version such that $\partial_{\mu} f' (t, x, \mu) (v)$ is uniformly bounded in $(t, x, \mu, v)$, Lipschitz continuous in $(x, \mu, v)$ uniformly in $t$ and $\alpha$-H\"older continuous in $t$ uniformly in $(x, \mu, v)$. 
  \end{enumerate}
    Moreover, assume that $f$ is controlled in time by $f'$ in the following way: 
  \begin{enumerate} \renewcommand{\labelenumi}{(\alph{enumi})} \setcounter{enumi}{8}
      \item
        \[ \sup_{x \in \mathbb{R}^d, \mu \in \mathcal{P}_2 (\mathbb{R}^d)} \sup_{s < t}  \frac{| f (t, x, \mu) - f (s, x, \mu) - f' (s, x, \mu) \delta W_{s, t} |}{| t - s |^{2 \alpha}} < + \infty ; \]
        
        \item
        \[ \sup_{x \in \mathbb{R}^d, \mu \in \mathcal{P}_2 (\mathbb{R}^d)}  \sup_{s < t}  \frac{| D_x f (t, x, \mu) - D_x f (s, x, \mu) - D_x
           f' (s, x, \mu) \delta W_{s, t} |}{| t - s |^{2 \alpha}} < + \infty ; \]
        \item
        \[ \sup_{x,v \in \mathbb{R}^d, \mu \in \mathcal{P}(\mathbb{R}^d)} \sup_{s < t}  \frac{| \partial_{\mu} f (t, x, \mu) (v) - \partial_{\mu} f (s, x, \mu) (v) - \partial_{\mu} f' (s, x, \mu) (v) \delta W_{s, t} |}{| t - s |^{2 \alpha}} < + \infty ; \]
  \end{enumerate}
\end{assumptions}

\noindent As an example, $(f(t, x,\mu),f'(t, x,\mu)) := (\int_{\R^d} g_t(x,y) \, \mu(dy) , \int_{\R^d} g'_t(x,y) \, \mu(dy))$ for some deterministic controlled vector field $(g,g') \in \mathscr{D}_W^{2\alpha} C_b^3(\R^d \times \R^d; \mathscr{L}(\R^n,\R^d))$, and $(f(t, x,\mu),f'(t, x,\mu)) : = (\phi_t(x,\int_{\R^d} y \, \mu(dy)) , \phi'_t(x,\int_{\R^d} y \, \mu(dy)))$ for $(\phi,\phi') \in \mathscr{D}_W^{2\alpha} C_b^3(\R^d \times \R; \mathscr{L}(\R^n,\R^d))$, both satisfy \cref{assumptions_existence}.
Observe also that, by considering $f$ as time independent and setting $f'=0$, \cref{assumptions_existence} recover \cite[Regularity assumptions 1 and 2]{BCD19}.

\begin{prop} \label{prop:liftofstochasticcontrolledvectorfields}
    Let $(\Omega,\mathcal{F},\mathbb{P})$ be an atomless probability space.
    Define $\hat{b}:[0,T] \times \R^d \times L_2(\Omega;\R^d) \to \R^d$ by $\hat b_t (x,X) := b(t,x,\mathcal{L}_X)$, and similarly for $\hat \sigma, \hat f$ and $\hat f'$. 
    Let \cref{assumptions_existence} be in force. Then, for any $p \in (2,\infty)$, the following hold: \begin{enumerate}
        \item for any $t \in [0,T]$, 
        $$\hat b_t (\cdot , \cdot) \in C^1_{b, \R^d \times L_p(\Omega;\R^d)}(\R^d \times L_2(\Omega;\R^d);\R^d)$$ and $\sup_{t \in [0,T]} (|\hat b_t (\cdot , \cdot)|_{\infty} + |\hat b_t (\cdot , \cdot)|_{Lip;\R^d \times L_p(\Omega;\R^d)}) < +\infty$, with similar properties for $\hat \sigma$; 
    \item  \begin{equation*}
        (\hat{f},\hat{f}') \in \mathscr{D}_W^{2\alpha} C_{b,\R^d \times L_p(\Omega;\R^d)}^{p\wedge 3} (\R^d \times L_2(\Omega;\R^d); \mathscr{L}(\R^n,\R^d)) .
    \end{equation*}
    \end{enumerate}
\end{prop}

\begin{proof}
    The proof of \textit{1.}\ is straightforward from the assumptions on $b$ and $\sigma$ and recalling that $\mathcal{W}_2 \le \mathcal{W}_p$ for $p \in (2,\infty)$. 
    Let us focus on \textit{2.} 
    For any fixed $t \in [0,T]$, the functions $\hat{f}_t: \R^d \times L_2(\Omega;\R^d) \to \mathscr{L}(\R^n,\R^d)$ and $\hat{f}'_t: \R^d \times L_2(\Omega;\R^d) \to \mathscr{L}(\R^n \otimes \R^n,\R^d)$ are continuously Fréchet differentiable. 
    Indeed, we necessarily have \begin{equation*}
        D_1 \hat{f}_t (x,X) = D_x f(t, x,\mathcal{L}_X)
    \end{equation*}
    and \begin{equation*}
        D_2 \hat{f}_t (x,X) [Y] = \E (\partial_\mu f(t, x,\mathcal{L}_X)(X) \cdot Y) \qquad \text{for any $Y \in L_2(\Omega;\R^d)$}.  
    \end{equation*}
    By (a) in \cref{assumptions_existence}, the first map is Lipschitz continuous as a function $\R^d \times L_2(\Omega;\R^d) \to \mathscr{L}(\R^d,\mathscr{L}(\R^n,\R^d))$, while from (c) we get the Lipschitz continuity of $D_2 \hat{f}_t$ as a map $\R^d \times L_2(\Omega;\R^d) \to \mathscr{L}(L_2(\Omega;\R^d),\mathscr{L}(\R^n,\R^d))$.
    Note that Lipschitz continuity on $\R^d \times L_2(\Omega;\R^d)$ implies Lipschitz continuity along $\R^d \times L_p(\Omega;\R^d)$ for any $p \in (2,\infty)$. 
    Arguing in a very similar way, from assumptions (g) and (h) it is possible to deduce that $\hat{f}'_t$ is continuously Fréchet differentiable with \begin{align*}
        D_1 \hat{f}'_t (x,X) &= D_x f'(t, x,\mathcal{L}_X) \\
        D_2 \hat{f}'_t (x,X) [Y] &= \E (\partial_\mu f'(t, x,\mathcal{L}_X)(X) \cdot Y) \qquad \text{for any $Y \in L_2(\Omega;\R^d)$},
    \end{align*}
    with $\R^d \times L_2(\Omega;\R^d) \ni (x,X) \mapsto D\hat{f}'_t(x,X) \in \mathscr{L}(\R^d \times L_2(\Omega;\R^d),\mathscr{L}(\R^n,\R^d))$ Lipschitz continuous. 

    From (a) we also deduce that there exists $D^2_1 \hat{f}_t (x,X) = D_x^2 f(t, x,\mathcal{L}_X)$ and it is continuous as a function $\R^d \times L_2(\Omega;\R^d) \to \mathscr{L}^2(\R^d , \mathscr{L}(\R^n,\R^d)))$. By respectively applying (b) and (f), for any $(y,Y) \in \R^d \times L_2(\Omega;\R^d)$ we get that \begin{align*}
        D_2 D_1 \hat{f}_t (x,X) [Y,y] &= \E(\partial_\mu D_x f (t, x,\mathcal{L}_X) (X) \cdot (Y \otimes y)) \\
        D_1 D_2 \hat{f}_t (x,X) [y,Y] &= \E(D_x \partial_\mu f (t, x,\mathcal{L}_X) (X) \cdot (y \otimes Y))
    \end{align*}
    and the two previous maps are Lipschitz continuous as maps from $\R^d \times L_2(\Omega;\R^d)$ to $\mathscr{L}(L_2(\Omega;\R^d) \otimes \R^d,\mathscr{L}(\R^n,\R^d))$ and $\mathscr{L}(\R^d \otimes L_2(\Omega;\R^d),\mathscr{L}(\R^n,\R^d))$, respectively. 
    A posteriori, once twice continuous Fréchet differentiability along $\R^d \times L_p(\Omega;\R^d)$ is established, Schwarz's theorem ensures that the two mixed derivatives coincide. 

    Showing the twice continuous partial Fréchet differentiability of $\hat{f}_t$ in its second entry requires extra care (see, for instance, \cite[Remark 5.80]{CARMONADELARUE_volumeI} to understand why this may be an issue). Notice that, due to \cref{assumptions_existence} and for any $x \in \R^d$, the function $\mathcal{P}_2(\R^d) \ni \mu \mapsto f(t, x,\mu)$ is fully $\mathcal{C}^2$-regular in the sense of \cite[Section 5.6.2, p. 464]{CARMONADELARUE_volumeI}, and we have uniform global bounds on the second derivatives. Taking into account \cite[Proposition 5.85]{CARMONADELARUE_volumeI} we get that the function $L_2(\Omega;\R^d) \ni X \mapsto D_2 \hat{f}_t(x,X)$ is continuously Gateaux differentiable, and \cite[Remark 5.86]{CARMONADELARUE_volumeI} underlines how its Gateaux differential is linear in the direction variable. This is enough to conclude that $\hat{f}_t(x,\cdot)$ is twice Fréchet differentiable with 
    \begin{align*}
        D^2_2 \hat f_t(x,X)[Y_1,Y_2] &= \E\left(D_v\partial_\mu f(t,x,\mathcal L_X)(X)\cdot [Y_1\otimes Y_2]\right) + \E\E'\left(\partial_\mu^2 f(t,x,\mathcal L_X)(X,\slashed X)\cdot [Y_1\otimes \slashed{Y_2}] \right) \\
        &= \E\big(D_v \partial_\mu f(t,x,\mathcal{L}_X)(X)\cdot [Y_1 \otimes Y_2]\big)\\
        & \quad + \E\left(\int_{\R^d \times \R^d}\partial^2_\mu f(t,x,\mathcal{L}_X)(X,\tilde x)\cdot [Y_1 \otimes \tilde y]\,\mathcal{L}_{(X,Y_2)}(d\tilde x,d\tilde y)\right),
    \end{align*}
    where $\E'$ denotes the expectation with respect to the auxiliary probability space $(\Omega',\mathcal{F}',\mathbb{P}')$ (cf.\ \hyperref[fact]{Fact} in \cref{section:McKean--ValsovroughSDEs} also for the definitions of $\slashed{X}$ and $\slashed{Y_2}$). 

    From (d) and (e), and by a standard application of H\"older's inequality, it is possible to show that, for any $p \in (2,\infty)$, the map \begin{equation*}
        \R^d \times L_2(\Omega;\R^d) \ni X \mapsto D^2_2 \hat{f}_t(x,X) \in \mathscr{L}^2(L_p(\Omega;\R^d), \mathscr{L}(\R^n,\R^d))
    \end{equation*}
    is $(p \wedge 3)$-H\"older continuous along $\R^d \times L_p(\Omega;\R^d)$ in the sense of \cref{def:HoldercontinuityalongsubBanachspaces}.
    Indeed, this essentially follows the same reasoning as the computations in \cref{rmk:counterexampleCD}, so we omit the details here. 

    So far, taking into account the global boundedness requirements in \cref{assumptions_existence}, we have proved that, for any $t \in [0,T]$ and for any $p \in (2,\infty)$, \begin{equation*}
        \hat{f}_t \in C^{p\wedge 3}_{b,\R^d \times L_p(\Omega;\R^d)}(\R^d \times L_2(\Omega;\R^d);\mathscr{L}(\R^n,\R^d)) 
    \end{equation*} and \begin{equation*}
        \hat{f}'_t \in C^2_{b,\R^d \times L_p(\Omega;\R^d)}(\R^d \times L_2(\Omega;\R^d);\mathscr{L}(\R^n \otimes \R^n,\R^d)). 
    \end{equation*}
    By construction, $\sup_{t \in [0,T]} (|\hat{f}_t|_{p \wedge 3;\R^d \times L_p(\Omega;\R^d)} + |\hat{f}'_t|_{2;\R^d \times L_p(\Omega;\R^d)}) < +\infty$. 
    Moreover, for any $g \in \{\delta \hat{f}, \delta D_x\hat{f}, \delta D^2_x\hat{f},\delta \hat{f}',\delta D_x \hat{f}'\}$, \begin{equation*}
        \sup_{0 \le s < t \le T} \frac{\sup_{(x,X) \in \R^d \times L_2(\Omega;\R^d)} |g_{s,t}(x,X)|}{|t-s|^\alpha} < +\infty ,
    \end{equation*}
    and from (i), (j) and (k) it is straighforward to deduce \begin{equation*}
        \sup_{0 \le s < t \le T} \frac{\sup_{(x,X) \in \R^d \times L_2(\Omega;\R^d)} |\delta\hat{f}_{s,t}(x,X) - \hat{f}'_s(x,X) \delta W_{s,t}|}{|t-s|^{2\alpha}} < +\infty 
    \end{equation*} and \begin{equation*}
        \sup_{0 \le s < t \le T} \frac{\sup_{(x,X) \in \R^d \times L_2(\Omega;\R^d)} |\delta D\hat{f}_{s,t}(x,X) - D\hat{f}'_s(x,X) \delta W_{s,t}|}{|t-s|^{2\alpha}} < +\infty .
    \end{equation*}
\end{proof}

Let $\mathbf{\Omega}=(\Omega,\mathcal{F},\{\mathcal{F}_t\}_{t \in [0,T]},\mathbb{P})$ be a 
Polish probability space endowed with a filtration such that $\mathcal{F}_0$ {is atomless and} contains all the $\mathbb{P}$-null sets, and let $B$ be an $m$-dimensional $\{\mathcal{F}_t\}_t$-Brownian motion on $\mathbf{\Omega}$. 
Let $\nu \in \mathcal{P}_2(\R^d)$ and let $\xi$ be any $\mathcal{F}_0$-measurable random variable on $\mathbf{\Omega}$ such that $\mathcal{L}_\xi = \nu$. {Thanks to the atomlessness of $\mathcal{F}_0$, such a random variable can always be constructed.}

\begin{thm} \label{thm:existence_meanfieldroughPDE}
   Let \cref{assumptions_existence} hold.
   For any $t \in [0,T]$ define $\mu_t := \mathcal{L}_{X_t}$, where $X$ is an $L_{4,\infty}$-integrable solution to the following McKean-Vlasov rough SDE on $\mathbf{\Omega}$, in the sense of \cref{def:solution_McKeanVlasovroughSDE}:
   \begin{equation} \label{eq:McKean-VlasovroughSDE_existence}
        \begin{cases}
            dX_t &= b(t, X_t,\mu_t) \, dt + \sigma(t, X_t,\mu_t) \, dB_t + (f,f')(t, X_t,\mu_t) \, d\mathbf{W}_t \qquad t \in [0,T] \\
            \mu_t &= \mathcal{L}_{X_t} \\
            X_0 &= \xi 
        \end{cases} .
    \end{equation}  
   Then $\mu$ is a solution to equation \eqref{eq:meanfieldroughPDEs} in the sense of \cref{def:solution_meanfieldroughPDE} starting from $\nu$.  
\end{thm}

\begin{proof}
    In view of \cref{prop:liftofstochasticcontrolledvectorfields}, take $p=4$ and consider the (unique) $L_{4,\infty}$-integrable solution $X$ to the rough SDE \eqref{eq:McKean-VlasovroughSDE_existence} on $\mathbf{\Omega}$, which is well-posed by \cref{thm:wellposedness_McKeanVlasovroughSDEs}.    
    Recall that $p \wedge 3 = 3 > \frac{1}{\alpha}$ by assumption. 
     For definiteness, we are choosing $(\Omega',\mathcal{F}', \{\mathcal{F}'_t\}_t, \mathbb{P}') = (\Omega, \mathcal{F},\{\mathcal{F}_t\}_t, \mathbb{P})$.
    Consider the Borel measurable \footnote{Being $X$ $\mathbb{P}$-almost surely continuous, the map $[0,T] \ni t \mapsto \langle \mu_t , \varphi \rangle \in \R$ is Borel measurable (in fact, continuous) for any $\varphi \in C^0_b(\R^d;\R)$. It is therefore possible to deduce that, for any $A \in \mathcal{B}(\R^d)$, the function $[0,T] \ni t \mapsto \mu_t(A) \in \R$ is Borel measurable. By \cite[Proposition 5.7]{CARMONADELARUE_volumeI}, that is equivalent to say that $[0,T] \ni t \mapsto \mu_t \in \mathcal{P}_2(\R^d)$ is Borel measurable.} map \begin{equation} \label{eq:definitionofmu}
        [0,T] \ni t \mapsto  \mu_t := \mathcal{L}_{X_t} \in \mathcal{P}_2(\R^d).
    \end{equation} 
    Notice that $\mu_0 = \nu$, and let $\varphi \in C_b^3(\R^d;\R)$ be arbitrary. We now show that $\mu$ satisfies all the requirements of \cref{def:solution_meanfieldroughPDE}.
    By assumptions, the map \begin{equation*}
        \begin{split}
            [0,T] \ni t &\mapsto \langle \mu_t , L_t[\mu_t] \varphi \rangle = \E \left( \frac{1}{2} \partial^2_{ij} \varphi(X_t) a^{ij}(t,X_t,\mu_t) + \partial_i \varphi(X_t) b^i(t,X_t,\mu_t) \right)
        \end{split}
    \end{equation*}
    is measurable and uniformly bounded. 
    Note that $X$ is a rough It\^o process in the sense of \cite[Section 4.4.2]{FHL21}, and recall that $$(\hat f(X, \slashed{X}), D_1\hat{f}(X,\slashed{X}) [\hat{f}(X,\slashed{X})]  +D_2\hat{f}(X,\slashed{X}) [\cancel{\hat{f}(X,\slashed{X})}] + \hat f' (X, \slashed{X})) \in \mathbf{D}_W^{2\alpha} L_{4,\infty}([0,T], \mathbf{\Omega}; \mathscr{L}(\R^n,\R^d))$$ and $\hat f$ is globally bounded by assumption.  By means of the rough It\^o formula (cf.\ \cite[Theorem 4.13]{FHL21} with $\gamma = 3$, $\beta = \beta' = \alpha$ and $n=\infty$) and taking into account \cref{rmk:FréchetderivativesinDavieexpansion}, for any $0 \le s \le t \le T$ we have  \begin{equation}
        \begin{aligned} \label{eq:Itoformula_existence}
        \varphi(X_t) - \varphi(X_s) &= \int_s^t \partial_i \varphi(X_r) b^i(r,X_r,\mu_r) + \frac{1}{2} \partial^2_{ij} \varphi(X_r) a^{ij}(r,X_r,\mu_r) \, dr  \\
        &\quad + \int_s^t \partial_i \varphi(X_r) \sigma^i_l(r,X_r,\mu_r) \, dB^l_r  + \partial_i \varphi(X_s) f^i_\kappa(s,X_s,\mu_s) \delta W^\kappa_{s,t}   \\
        & \quad + \Big( \partial^2_{ij} \varphi(X_s) f^i_\lambda(s,X_s,\mu_s) f^j_\kappa(s,X_s,\mu_s) +  \partial_i \varphi(X_s) \partial_{x^j} f^i_\lambda(t,y,\nu) f^j_\kappa(t,y,\nu)  \\ 
         & \quad + \int_{\R^d} \partial_\mu f^i_\lambda(t,y,\nu)(z) \cdot f_\kappa(t,z,\nu) \, \nu(dz) + \partial_i \varphi(y)(f_{\lambda \kappa}')^i(t,y,\nu) \Big) \mathbb{W}^{\kappa\lambda}_{s,t} + \varphi(X)_{s,t}^\natural ,
        \end{aligned}
    \end{equation}  
    where $\|\E(\varphi(X)_{s,t}^\natural \mid \mathcal{F}_s)\|_\infty \lesssim |t-s|^{3\alpha}$ with uniform implicit constant over bounded sets of $C^3_b(\R^d;\R)$. 
    By taking the expectation in \eqref{eq:Itoformula_existence} with respect to $\mathbb{P}$ and by the martingale property of It\^o's integrals, we conclude that  \begin{equation*}
        \begin{aligned}
            &\langle \mu_t , \varphi \rangle - \langle \mu_s , \varphi \rangle = \E(\varphi(X_t) - \varphi(X_s))  \\
            &= \int_s^t \langle \mu_r, L_r[\mu_r] \varphi \rangle \, dr + \langle \mu_s, \Gamma_s[\mu_s]_\kappa \rangle \delta W^\kappa_{s,t}  + \langle \mu_s, (\Gamma_s[\mu_s]_\kappa \Gamma_s[\mu_s]_\lambda + \Gamma'_s[\mu_s]_{\lambda \kappa} ) \varphi  \rangle \mathbb{W}^{\kappa\lambda}_{s,t} + \E(\varphi(X)^\natural_{s,t})
        \end{aligned}
    \end{equation*}
    with $|\E(\varphi(X)^\natural_{s,t})| \le \|\E(\varphi(X)^\natural_{s,t} \mid \mathcal{F}_s)\|_\infty = |t-s|^{3\alpha}$. 
    We are left to prove that condition (ii) of \cref{def:solution_meanfieldroughPDE} is satisfied. 
    Again from \cite[Theorem 4.13]{FHL21}, we also have that    
    \begin{equation*}
        (Y^\varphi,(Y^\varphi)'):=( D\varphi(X) [f(X,\mu)] , D^2 \varphi(X)[f(X,\mu), f(X,\mu)] + D\varphi(X) [\Phi(X,\mu)] ) 
    \end{equation*}
    belongs to $\mathbf{D}_W^{2\alpha}L_{2} ([0,T],\mathbf{\Omega};\R^n)$, where  \begin{multline*}
        \Phi_{\kappa\lambda} (t, X_t, \mu_t) := \partial_j f_\lambda(t,X_t, \mu_t) f^j_\kappa(t,X_t, \mu_t)  \\ + \int_{\R^d} \partial_\mu f_\lambda(t,X_t, \mu_t)(z) \cdot f_\kappa(t,z,\mu_t) \, \mu_t(dz) + f_{\lambda \kappa}'(t,X_t, \mu_t) 
    \end{multline*}
    for any $\kappa,\lambda=1,\dots,n$ and $\|\delta Y^\varphi\|_{\alpha;2} + \sup_{t \in [0,T]}\|(Y^\varphi)_t'\|_2 + \|\delta (Y^\varphi)'\|_{\alpha;2} + \|\E_\cdot R^Y\|_2 < +\infty$ uniformly over bounded sets of $\varphi \in C^3_b(\R^d;\R)$.  
    Hence it follows that \begin{multline*}
        | \langle  \mu_t, (\Gamma_t[\mu_t]_\kappa \Gamma_t[\mu_t]_\lambda + \Gamma'_t[\mu_t]_{\kappa\lambda} ) \varphi  \rangle - \langle \mu_s, (\Gamma_s[\mu_s]_\kappa \Gamma_s[\mu_s]_\lambda + \Gamma'_s[\mu_s]_{\kappa\lambda} ) \varphi  \rangle |  \\ 
        \le |\E ((Y^\varphi)'_t - (Y^\varphi)'_s)| \le \|\delta (Y^\varphi)'\|_{\alpha;2} |t-s|^\alpha
    \end{multline*} and \begin{multline*}
        |\delta  \langle \mu_\cdot, \Gamma_\cdot[\mu_\cdot]_\kappa \varphi \rangle_{s,t} - \langle \mu_s, (\Gamma_s[\mu_s]_\kappa \Gamma_s[\mu_s]_\lambda + \Gamma'_s[\mu_s]_{\kappa\lambda} ) \varphi  \rangle \delta W^\lambda_{s,t} |  \\
        = |\E( (Y^\varphi_t)^\kappa - (Y^\varphi_s)^\kappa - \sum_{\lambda=1}^n ((Y^\varphi)'_s)^{\kappa \lambda} \delta W^\lambda_{s,t} )| \le \|\E( R^{Y^\varphi}_{s,t} \mid \mathcal{F}_s  )\|_2 \lesssim |t-s|^{2\alpha}. 
    \end{multline*}
\end{proof}


\begin{rmk}
    In \cref{thm:existence_meanfieldroughPDE}, we have chosen an $L_{p,\infty}$-integrable solution $X$ to \eqref{eq:McKean-VlasovroughSDE_existence} with $p=4$ purely by convention; as the proof shows, this choice aligns naturally with the rough Itô formula presented in \cite[Theorem 4.13]{FHL21}. 
    Let $\bar p \in (\frac{1}{\alpha},\infty)$ and let $\bar X$ be an $L_{\bar p,\infty}$-integrable solution to \eqref{eq:McKean-VlasovroughSDE_existence} starting from $\xi$. 
    Notice that both $X$ and $\bar X$ are $L_{\bar p \wedge 4,\infty}$-integrable solutions to \eqref{eq:McKean-VlasovroughSDE_existence}.
    By the stability estimates for $L_{\bar p \wedge 4}$-integrable solutions in \cite[Theorem 3.11]{FHL25},  we can conclude that, uniformly in $t \in [0,T]$, \begin{equation*}
        \mathcal{W}_2(\mathcal{L}_{X_t}, \mathcal{L}_{\bar X_t}) \le \mathcal{W}_{\bar p \wedge 4}(\mathcal{L}_{X_t}, \mathcal{L}_{\bar X_t}) \le \big\| \sup_{s \in [0,T]} |X_s - \bar X_s| \big\|_{\bar p \wedge 4} \le 0. 
    \end{equation*}
    Hence, $\mu=(\mu_t)_{t \in [0,T]}$ defined as in \cref{thm:existence_meanfieldroughPDE} describes the curve of one-dimensional marginals of any $L_{p,\infty}$-integrable solution to \eqref{eq:McKean-VlasovroughSDE_existence} starting from $\xi$ with $p \in (\frac{1}{\alpha}, \infty)$. 
\end{rmk}

\subsection{Uniqueness} \label{section:uniqueness}
{
To prove uniqueness for \eqref{eq:meanfieldroughPDEs}, we show that any solution must necessarily be the law of the corresponding McKean--Vlasov rough SDE. While our methodology is novel in the context of \eqref{eq:meanfieldroughPDEs}, the conceptual strategy itself is not new, as it was previously employed by \cite{CG19} and \cite{CN21}.
This uniqueness argument relies on two key ingredients: the well-posedness of the McKean--Vlasov RSDE (cf.\ \cref{thm:wellposedness_McKeanVlasovroughSDEs}), and a duality-based uniqueness result for linear rough PDEs (cf.\ \cref{prop:duality} below).
In order for the duality argument to hold, it is necessary to assume extra regularity for the coefficients in the state variable $x$. 
}

\begin{assumptions} \label{assumptions_uniqueness}
    Let $b:[0,T] \times \R^d \times \mathcal{P}_2(\R^d) \to \R^d$, $\sigma: [0,T] \times \R^d \times \mathcal{P}_2(\R^d) \to \mathscr{L}(\R^m,\R^d)$, $f:[0,T] \times \R^d \times \mathcal{P}_2(\R^d) \to \mathscr{L}(\R^n,\R^d)$ and $f':[0,T] \times \R^d \times \mathcal{P}_2(\R^d) \to \mathscr{L}(\R^n \otimes \R^n ,\R^d)$ be Borel measurable. Suppose that
    \begin{enumerate} \renewcommand{\labelenumi}{(\alph{enumi})}

        \item for any $(t,\mu) \in [0, T] \times \mathcal{P}_2 (\mathbb{R}^d)$ and for any $g \in \{b,\sigma\}$, the map $x \mapsto g (t, x, \mu)$ is 2-times differentiable in classical sense. Moreover, the functions \[ (t, x, \mu) \mapsto (g (t, x, \mu), D_x g (t, x, \mu), D^2_x g (t, x, \mu)) \] are globally bounded, Lipschitz continuous in $(x, \mu)$ uniformly in $t$ and continuous in $t$ uniformly in $(x, \mu)$ ;

    
        \item for any $(t,\mu) \in [0, T] \times \mathcal{P}_2 (\mathbb{R}^d)$, the map $x \mapsto f (t, x, \mu)$ is 4-times differentiable in classical sense. Moreover, the functions \[ (t, x, \mu) \mapsto (f (t, x, \mu), D_x f (t, x, \mu), D^2_x f (t, x, \mu), D^3_x f (t, x, \mu), D^4_x f (t, x, \mu)) \] are globally bounded, Lipschitz continuous in $(x, \mu)$ uniformly in $t$ and $\alpha$-H\"older continuous in $t$ uniformly in $(x, \mu)$ ;
  
        \item for any $(t,\mu) \in [0, T] \times \mathcal{P}_2 (\mathbb{R}^d)$, the map $x \mapsto f' (t, x, \mu)$ is 3-times differentiable in classical sense. Moreover, the functions \[ (t, x, \mu) \mapsto (f' (t, x, \mu), D_x f' (t, x, \mu), D^2_x f' (t, x, \mu), D^3_x f' (t, x, \mu)) \] are globally bounded, Lipschitz continuous in $(x, \mu)$ uniformly in $t$ and $\alpha$-H\"older continuous in $t$ uniformly in $(x, \mu)$ ;
     
        \item for any $(t,\mu) \in [0, T] \times \mathcal{P}_2 (\mathbb{R}^d)$ and for any $v \in \mathbb{R}^d$, the map $x \mapsto \partial_{\mu} f (t, x, \mu) (v)$ is 3-times differentiable in classical sense. Moreover, the functions \[ (t, x, \mu, v) \mapsto (\partial_{\mu} f (t, x, \mu) (v), D_x \partial_{\mu} f (t, x, \mu) (v), \ldots, D^3_x \partial_{\mu} f (t, x, \mu) (v)) \] are globally bounded, Lipschitz continuous in $(x, \mu, v)$ uniformly in $t$ and $\alpha$-H\"older continuous in $t$ uniformly in $(x, \mu, v)$ ;
  
        \item \[ \sup_{\mu \in \mathcal{P}_2 (\mathbb{R}^d)} \sup_{s < t} \frac{| f (t, \cdot, \mu) - f (s, \cdot, \mu) - f' (s, \cdot, \mu) \delta W_{s, t} |_{C^3_b}}{| t - s |^{2 \alpha}} < + \infty . \]
        
\end{enumerate}
\end{assumptions}

\begin{thm} \label{thm:uniqueness_meanfieldroughPDEs}
    Let $\nu \in \mathcal{P}_2(\R^d)$. Let \cref{assumptions_existence} hold and let $b,\sigma$ as in \cref{assumptions_uniqueness}. 
    Then there is at most one continuous solution $\mu :[0,T] \to \mathcal{P}_2(\R^d)$ to \eqref{eq:meanfieldroughPDEs} starting from $\nu$ and satisfying the following two conditions: 
    \begin{enumerate}
        \item for any $R>0$ there is a constant $C_R >0$ such that $|\langle \mu_t - \mu_s, \phi\rangle| \le C_R |t-s|^\alpha$ for any $\phi \in C^1(\R^d;\R)$ with $|\phi|_{Lip} \le R$; 
        \item 
        $(f^\mu,(f^\mu)') \in \mathscr{D}_W^{2\alpha} C_b^5(\R^d;\mathscr{L}(\R^n,\R^d))$,
        where 
        \begin{align*}
            f^\mu_t(x) &:= f(t, x,\mu_t) \\
            (f^\mu)'_t(x) &:= \int_{\R^d} \partial_\mu f (t, x,\mu_t) (z) \cdot f(t, z,\mu_t) \, \mu_t(dz) + f'(t, x,\mu_t).
        \end{align*}
    \end{enumerate} 
\end{thm}

\begin{proof}
    Let $\nu \in \mathcal{P}_2(\R^d)$ and let $\mu:[0,T] \to \mathcal{P}_2(\R^d)$ be a continuous solution to \eqref{eq:meanfieldroughPDEs} starting from $\nu$ and satisfying \textit{1.}\ and \textit{2}. 
    For any $(t,x) \in [0,T] \times \R^d$, define $b^\mu_t(x) := b(t, x,\mu_t)$. It is straightforward to see that, under \cref{assumptions_uniqueness} and by continuity of $\mu$, $b^\mu$ is a (deterministic) bounded Lipschitz vector field from $\R^d$ to $\R^d$ in the sense of \cite[Definition 4.1]{FHL21} and that $[t \mapsto b^\mu(t,\cdot)] \in C^0([0,T]; C^2_b(\R^d;\R^d)) \cap \mathrm{Bd}([0,T];C^3_b(\R^d;\R^d)$. 
    Indeed, note that, for any $(t,\mu)$, $x \mapsto b^\mu_t(x)$ is 2-times continuously differentiable and $x \mapsto D^2 b^\mu_t(x)$ is Lipschitz continuous. Moreover, \begin{equation*}
        |b^\mu_t(x) - b^\mu_s(x)| \le \sup_{(x,\mu) \in \R^d \times \mathcal{P}_2(\R^d)} |b(t,x,\mu) - b(s,x,\mu)| + |b(s,x,\cdot)|_{Lip} \mathcal{W}_2(\mu_t, \mu_s) \to 0
    \end{equation*}
    as $|t-s| \to 0$, and the same holds replacing $b$ by $Db$ or $D^2b$. 
    Similar conditions also hold for $\sigma^\mu_t(x) := \sigma_t(x,\mu_t)$.
    Let $\mathbf{\Omega}=(\Omega,\mathcal{F},\{\mathcal{F}_t\}_{t \in [0,T]},\mathbb{P})$ be any Polish probability space endowed with a filtration such that $\mathcal{F}_0$ {is atomless and} contains all the $\mathbb{P}$-null sets, and let $B$ be an $m$-dimensional $\{\mathcal{F}_t\}_t$-Brownian motion on $\mathbf{\Omega}$. 
    Let $\xi$ be any $\mathcal{F}_0$-measurable random variable on $\mathbf{\Omega}$ such that $\mathcal{L}_\xi = \nu$.
    Denote by $X^\mu$ be the unique $L_{4,\infty}$-integrable solution to the following rough SDE on $\mathbf{\Omega}$ : \begin{equation} \label{eq:roughSDE_uniqueness}
        X^\mu_t = \xi + \int_0^t b_r^\mu(X_r^\mu) \, dr + \int_0^t \sigma^\mu_r(X_r^\mu) \, dB_r + \int_0^t (f^\mu_r, (f^\mu)_r')(X_r^\mu) \, d\mathbf{W}_r , \quad t \in [0,T] .
    \end{equation}
    Note that, because of condition \textit{2.}, equation  \eqref{eq:roughSDE_uniqueness} is well-posed by \cite[Theorem 4.7]{FHL21} as a standard rough SDE. 
    Necessarily, the curve $[0,T] \ni t \mapsto \mathcal{L}_{X^\mu_t} \in \mathcal{P}_2(\R^d)$ solves the following linear rough PDE \begin{equation*}
        d\rho_t = \left( \frac{1}{2} \partial^2_{ij} ((a^\mu)^{ij}(t,\cdot) \rho_t) - \partial_i ((b^\mu)^i(t,\cdot)\rho_t) \right) \, dt - \partial_i ((f^\mu)^i_\kappa(t,\cdot)\rho_t) d\mathbf{W}_t^\kappa
    \end{equation*}
    with initial condition $\nu$ and $a^\mu(t,x) := \sigma(t,x,\mu_t) \sigma(t,x,\mu_t)^\top$, in the sense of \cref{def:linearroughPDE_solution} below. The proof of this fact is analogous to the one of \cref{thm:existence_meanfieldroughPDE}, considering measure-independent coefficients. 
    By construction, $\mu$ solves the same linear rough PDE. Hence, by \cref{prop:duality} below and condition \textit{1.}, we can conclude that \begin{equation*}
        \mu_t = \mathcal{L}_{X^\mu_t} \qquad \text{for any $t \in [0,T]$} .
    \end{equation*}
    The assertion is proved if we show that $X^\mu$ is in fact the unique solution to a McKean--Vlasov rough SDE of the form \eqref{eq:solution_McKeanVlasovroughSDE}. 
    Let $\hat{b}:[0,T] \times \R^d \times L_2(\Omega;\R^d) \to \R^d$ be the Lions lift of $b$ with respect to $(\Omega,\mathcal{F},\mathbb{P})$. Namely, $\hat b_t(x,X) := b(t, x,\mathcal{L}_X)$ for any $X \in L_2(\Omega;\R^d)$. 
    In a very similar way we define $\hat{\sigma}$, $\hat{f}$ and $\hat{f}'$.     
    Notice that, being $X^\mu$ an $L_{4,\infty}$-integrable solution to \eqref{eq:roughSDE_uniqueness}, the following belongs to $\mathbf{D}_W^{2\alpha} L_{4,\infty}([0,T] , \mathbf{\Omega} ; \mathscr{L}(\R^n,\R^d))$:  
        \begin{multline*}
        (f^\mu(X^\mu), D_x f^\mu(X^\mu) f^\mu(X^\mu) + (f^\mu)'(X^\mu)) 
         \\ 
        = \big( \hat{f}(X^\mu, \cancel{X^\mu}) , D \hat{f} (X^\mu, \cancel{X^\mu}) [f(X^\mu, \cancel{X^\mu}), \cancel{f(X^\mu, \cancel{X^\mu})}] + \hat{f}'(X^\mu, \cancel{X^\mu}) \big) .
    \end{multline*}
    In the previous identity we made the natural choice $(\Omega',\mathcal{F}',\mathbb{P}') = (\Omega, \mathcal{F}, \mathbb{P})$. 
    Moreover, for any $0 \le s \le t \le T$, it holds \begin{equation*}
        \begin{split}
            \delta X^\mu _{s,t} &= \int_s^t b_r(X^\mu_r , \mathcal{L}_{X^\mu_r}) \, dr + \int_s^t \sigma_r(X^\mu_r , \mathcal{L}_{X^\mu_r}) \, dB_r +  f_s(X^\mu_s,\mathcal{L}_{X^\mu_s}) \delta W_{s,t}  \\
            & \quad  + \Big(D_x f_s(X^\mu_s,\mathcal{L}_{X^\mu_s}) f_s(X^\mu_s,\mathcal{L}_{X^\mu_s})   \\
            & \quad + \int_{\R^d} \partial_\mu f_s(X^\mu_s,\mathcal{L}_{X^\mu_s})(z) \cdot f_s(z,\mathcal{L}_{X^\mu_s}) \, \mathcal{L}_{X^\mu_s} (dz) + f_s'(X^\mu_s,\mathcal{L}_{X^\mu_s}) \Big) \mathbb{W}_{s,t} + X^{\mu,\natural}_{s,t}  \\
            &= \int_s^t \hat{b}_r(X^\mu_r , \cancel{X^\mu_r}) \, dr + \int_s^t \hat{\sigma}_r(X^\mu_r , \cancel{X^\mu_r}) \, dB_r +  \hat{f}_s(X^\mu_s,\cancel{X^\mu_s}) \delta W_{s,t}  \\
            & \quad + D\hat{f}_s(X^\mu_s,\cancel{X^\mu_s}) [\hat{f}_s(X^\mu_s,\cancel{X^\mu_s}), \cancel{\hat{f}_s(X^\mu_s,\cancel{X^\mu_s})}] \mathbb{W}_{s,t} + X^{\mu,\natural}_{s,t} ,
        \end{split}
    \end{equation*}
    where $\|\E(X^{\mu,\natural}_{s,t} \mid \mathcal{F}_s) \|_\infty \lesssim |t-s|^{3\alpha}$. Hence, we proved that $X^\mu$ is the unique $L_{4,\infty}$-integrable solution to \begin{equation} \label{eq:McKean-VlasovroughSDE_uniqueness}
    \begin{cases}
        dX_t &= b_t(X_t,\rho_t) \, dt + \sigma_t(X_t,\rho_t) \, dB_t + (f_t,f'_t)(X_t,\rho_t) \, d\mathbf{W}_t \\
        \rho_t &= \mathcal{L}_{X_t}
    \end{cases}
    \end{equation} 
    starting at $\xi$.
    Such a McKean--Vlasov rough SDE is well-posed under \cref{assumptions_existence}, as we pointed out in the proof of \cref{thm:existence_meanfieldroughPDE}. Hence, uniqueness of $\mu$ follows from uniqueness of solutions to \eqref{eq:McKean-VlasovroughSDE_uniqueness}.     
\end{proof}

As is typical in PDE theory, uniqueness of solutions to \eqref{eq:meanfieldroughPDEs} is established within a specific class of solutions, namely the class of continuous measure-valued paths that satisfy conditions \textit{1.}\ and \textit{2}.
The following result shows that the curve $\mu_t$ defined in \eqref{eq:definitionofmu} meets these two criteria, {thereby yielding existence of a solution within this class}.  
Indeed, it is sufficient to replace the process $Z^\mu$ by $X$ below, where $X$ is an $L_{4,\infty}$-integrable solution to the McKean--Vlasov rough SDE \eqref{eq:McKean-VlasovroughSDE_existence}. 
{
\cref{thm:existence_meanfieldroughPDE} and \cref{thm:uniqueness_meanfieldroughPDEs} can therefore be combined into a unified well-posedness result, as stated in \cref{thm:mainresult}.}

\begin{prop} \label{prop:stabilityundercomposition_uniqueness}
    Let $f$ and $f'$ as in \cref{assumptions_uniqueness}. 
    Let $\mu:[0,T] \to \mathcal{P}_2(\R^d)$ be such that there is a filtered probability space $\mathbf{\Omega} = (\Omega,\mathcal{F},\{\mathcal{F}_t\}_{t \in [0,T]},\mathbb{P})$ and an $\R^d$-valued stochastic process $Z^\mu=(Z^\mu_t)_{t \in [0,T]}$ on $\mathbf{\Omega}$ with $\mu_t = \mathcal{L}_{Z^\mu_t}$ for any $t \in [0,T]$ and $(Z^\mu,f(Z^\mu,\mu)) \in \mathbf{D}_W^{2\alpha}([0,T],\mathbf{\Omega};\R^d)$.
    Then the following hold: 
    \begin{enumerate} \renewcommand{\labelenumi}{(\roman{enumi})}
        \item $\mathcal{W}_2(\mu_t, \mu_s) \overset{\alpha}{=} 0$ ; 
        \item for any $R >0$ there exists $C_R >0$ such that $|\langle \mu_t - \mu_s, \phi \rangle| \le C_R |t-s|^\alpha$ for any $\phi \in C^1(\R^d; \R)$ with $|\phi|_{Lip} \le R$,  
        \item For any $(t,x) \in [0,T] \times \R^d$, define \begin{align*}
        g_t(x) &:= f(t, x,\mu_t) \\
        g'_t(x) &:= \int_{\R^d} \partial_\mu f (t, x,\mu_t) (z) \cdot f(t, z,\mu_t) \, \mu_t(dz) + f'(t, x,\mu_t) .
        \end{align*}
        Then $(g,g') \in \mathscr{D}_W^{2\alpha} C_b^5(\R^d;\mathscr{L}(\R^n,\R^d))$. 
    \end{enumerate}
\end{prop}

\begin{proof} \textit{(i). }  For any $s,t \in [0,T]$ it holds that \begin{equation*}
    \mathcal{W}_2(\mu_t, \mu_s) \le \|Z^\mu_t - Z^\mu_s\|_2 \le \|\delta Z\|_{\alpha;2} |t-s|^\alpha .
\end{equation*}

\textit{(ii). } For any $\phi \in C^1(\R^d;\R)$ with $|\phi|_{Lip} \le R$ and for any $s,t \in [0,T]$, \begin{equation*}
    |\langle \mu_t - \mu_s, \phi \rangle| \le \E (|\phi(Z^\mu_t) - \phi(Z^\mu_s)|) \le R \|Z^\mu\|_{\alpha;2} |t-s|^\alpha .
\end{equation*}

    \textit{(iii). } For any $t \in [0,T]$, the map $x \mapsto g_t(x)$ belongs to $C_b^5(\R^d;\mathscr{L}(\R^n,\R^d))$ by assumption (b). In fact, it belongs to $C_b^4$ and, for any $x,y \in \R^d$, \begin{equation*}
        |D^4_x g_t(x) - D^4_x g_t(y)| = |D^4_x f(t, x,\mu_t) - D^4_x f(t, y,\mu_t)| \le |D^4_x f(t,\cdot,\mu_t)|_{Lip} |x-y| .
    \end{equation*}
    Similarly, from (c) and (d) it is possible to prove that $\R^d \ni x \mapsto g'_t(x)$ belongs to $C_b^4(\R^d;\mathscr{L}(\R^n,\R^d))$. 
    In this regard, note that, under our assumptions, \begin{equation*}
        D_x \left(\int_{\R^d} \partial_\mu f_t (\cdot,\mu_t) (z) \cdot f_t(z,\mu_t) \, \mu_t(dz)\right)(x) =\int_{\R^d} D_x \partial_\mu f_t (x,\mu_t) (z) \cdot f_t(z,\mu_t) \, \mu_t(dz) .
    \end{equation*}
    For any $0 \le s \le t \le T$ and uniformly in $x \in \R^d$, we have \begin{equation*}
        \begin{split}
            |g_t(x) - g_s(x)| &= |f(t, x,\mu_t) - f(s, x,\mu_s)| \\
             &\le |f(t, x,\mu_t) - f(s, x,\mu_t)| + |f(s, x,\mu_t) - f(s, x,\mu_s)| \\
             & \le \sup_{(x,\mu) \in \R^d \times \mathcal{P}_2(\R^d)} |f(t, x,\mu) - f(s, x,\mu)| + \sup_{(t,x) \in [0,T] \times \R^d}|f(t,x,\cdot)|_{Lip} \mathcal{W}_2(\mu_t,\mu_s) \\
             &\lesssim |t-s|^\alpha .
        \end{split}
    \end{equation*}
    Analogous computations show that $\sum_{l=1}^4 |D^l_x g_t(x) - D^l_x g_s(x)|  \lesssim |t-s|^\alpha$ uniformly in $x \in \R^d$. \\
    %
    o
    For any $0 \le s \le t \le T$ and for any $x \in \R^d$, we can write \begin{multline} \label{eq:decompositionoftheremainder_uniqueness}
        g_t(x) - g_s(x) - g'_s(x) \delta W_{s,t} = f(t, x,\mu_t) - f(s, x,\mu_t) - f'(s, x,\mu_t) \delta W_{s,t}  \\
        + f(s, x,\mu_t) - f(s, x,\mu_s)  - \int_{\R^d} \partial_\mu f(s, x,\mu_s) (z) \cdot (f(s, z,\mu_s) \delta W_{s,t}) \, \mu_s(dz)  \\ + (f'(s, x,\mu_t)   - f'(s, x,\mu_s)) \delta W_{s,t} . 
    \end{multline}
    {Applying \cref{lemma:remainderinTaylorexpansion} to the map $\mu \mapsto f(s,x,\mu)$, for any coupling $\pi$ of $\mu_s$ and $\mu_t$ we have that \begin{multline*}
        \left| f(s, x,\mu_t) - f(s, x,\mu_s)  - \int_{\R^d} \partial_\mu f(s, x,\mu_s) (z) \cdot (f(s, z,\mu_s) \delta W_{s,t}) \, \mu_s(dz) \right| \\ \le \left| \int_{\R^d} \partial_\mu f(s, x,\mu_s) (z) \cdot (w-z-f(s, z,\mu_s) \delta W_{s,t}) \, \pi (dz,dw) \right|  \\ 
        + |\partial_\mu f(s, x,\cdot)|_{Lip} \int_{\R^d \times \R^d} |w-z|^2 \, \pi(dz,dw) .
    \end{multline*}
    Let us choose $\pi=\mathcal{L}_{(Z^\mu_s,Z^\mu_t)}$. Recalling that $(Z^\mu, f(Z^\mu,\mu))$ is a stochastic controlled rough path in the sense of \cref{def:stochasticcontrolledroughpath}, we can conclude that \begin{align*}
        \Big| & \int_{\R^d}  \partial_\mu f(s, x,\mu_s) (z) \cdot (w-z-f(s, z,\mu_s) \delta W_{s,t}) \, \mathcal{L}_{(Z^\mu_s,Z^\mu_t)}(dz,dw) \Big| \\ 
        &\quad = \left| \E( \partial_\mu f(s, x,\mu_s) (Z^\mu_s) \cdot \E(Z^\mu_t-Z^\mu_s-f(s, Z^\mu_s,\mu_s) \delta W_{s,t} \mid \mathcal{F}_s) ) \right| \\
        &\quad \le \sup_{t,x,\mu,v} |\partial_\mu f(t, x,\mu)(v)| \  \|\E( Z^\mu_t-Z^\mu_s-f(s, Z^\mu_s,\mu_s) \delta W_{s,t} \mid \mathcal{F}_s)\|_2 \lesssim |t-s|^{2\alpha}
    \end{align*}
    and \begin{equation*}
        |\partial_\mu f(s, x,\cdot)|_{Lip} \int_{\R^d \times \R^d} |w-z|^2 \, \mathcal{L}_{(Z^\mu_s,Z^\mu_t)}(dv,dw) = |\partial_\mu f(s, x,\cdot)|_{Lip} \|Z^\mu_t - Z^\mu_s\|_2^2 \lesssim |t-s|^{2\alpha},
    \end{equation*}
    where the implicit constants are uniform in $x \in [0,T]$.
    }
    Moreover, due to assumption (e), \begin{equation*}
        \sup_{x \in \R^d} |f(t, x,\mu_t) - f(s, x,\mu_t)  - f'(s, x,\mu_t) \delta W_{s,t}| \lesssim |t-s|^{2\alpha}
    \end{equation*} and, considering also (c), \begin{equation*}
        \sup_{x \in \R^d} |(f'(s, x,\mu_t)   - f'(s, x,\mu_s)) \delta W_{s,t}| \lesssim |t-s|^{2\alpha}
    \end{equation*}
    Notice that the decomposition in \eqref{eq:decompositionoftheremainder_uniqueness} can be repeated by replacing $g$ and $g'$ with $D_x g$ and $D_x g'$, respectively. (In fact the conclusion follows also for second and third order derivatives with respect to $x$.) To do so, assumption (d) is highly taken into account . 
    Hence we obtain that \begin{equation*}
        |g_t(\cdot) - g_s(\cdot) - g'_s(\cdot) \delta W_{s,t}|_{C^3_b} \lesssim|t-s|^{2\alpha} .
    \end{equation*}
\end{proof}

\section{Linear rough PDEs} \label{section:linearroughPDEs}

In this section we study well-posedness of the linear version of \eqref{eq:meanfieldroughPDEs}, namely in the case of measure-independent coefficients.
Our focus in on uniqueness, since it is the key ingredient to have uniqueness in the nonlinear case and it is used in the proof of \cref{thm:uniqueness_meanfieldroughPDEs}. 
Let $\mathbf{W}=(W,\mathbb{W}) \in \mathscr{C}_g^\alpha([0,T];\R^n)$ be a weakly geometric rough path with $\alpha \in (\frac{1}{3}, \frac{1}{2}]$. 
The equation we consider is of the form \begin{equation} \label{eq:solution_linearroughPDE}
    d\mu_t = \left(\frac{1}{2} \partial^2_{ij} (a^{ij}(t,\cdot) \mu_t) - \partial_i (b^i(t,\cdot) \mu_t) \right) \, dt - \partial_i (f^i(t,\cdot) \mu_t ) \, d\mathbf{W}_t
\end{equation} 
for some coefficients \begin{equation*}
    (b,\sigma,f,f'): [0,T] \times \R^d \to \R^d \times \mathscr{L}(\R^m,\R^d) \times \mathscr{L}(\R^n,\R^d) \times \mathscr{L}(\R^n \otimes \R^n,\R^d)
\end{equation*}
with $a = \sigma \sigma^\top$. 
The notion of solution is the one obtained by considering measure-independent coefficients in \cref{def:solution_meanfieldroughPDE} and the proof of the existence of a solution for \eqref{eq:solution_linearroughPDE} is completely analogous to the nonlinear case (see \cref{thm:existence_meanfieldroughPDE}). 

\begin{rmk}
    Let $\mu$ be a solution to \eqref{eq:meanfieldroughPDEs} in the sense of \cref{def:solution_meanfieldroughPDE}. 
    For $g \in \{b,\sigma,f,f'\}$ define $g^\mu(t,x) := g(t,x,\mu_t)$. 
    Hence $\mu$ solves the linearized equation 
    $$d\mu_t = \left(\frac{1}{2} \partial^2_{ij} ((a^\mu)^{ij}(t,\cdot) \mu_t) -\partial_i ((b^\mu)^i(t,\cdot)\mu_t) \right) \, dt -  \partial_i ((f^\mu)^i(t,\cdot) \mu_t) \, d\mathbf{W}_t.$$
    Substituting $\mu$ into the measure-valued component of the coefficients introduces an additional time dependence, which remains unavoidable even if the original coefficients were time-independent.
    Our theory on linear forward-in-time measure-valued rough PDEs can be therefore seen as an improvement of the time-independent theory in \cite{DFS17}. 
\end{rmk} 
 
We prove uniqueness for \eqref{eq:solution_linearroughPDE} by a duality argument. 
The dual equation on the time interval $[0,T]$ is a backward function-valued rough PDE of the form \begin{equation} \label{eq:backwardlinearroughPDE}
        - du_t(x) = \left( \frac{1}{2}  a^{ij}(t,x) \partial^2_{ij} u_t(x)  + b^i(t,x) \partial_i u_t(x) \right) \, dt + f^i(t,x) \partial_i u_t(x) \, d\mathbf{W}_t .
\end{equation}  

\begin{defn} \label{def:linearroughPDE_solution}
    (see \cite[Definition 4.1]{BFS24})
    A solution to \eqref{eq:backwardlinearroughPDE} is a map $u:[0,T] \times \R^d \to \R$ satisfying the following, for any $0 \le s \le t \le T$ and uniformly in $x \in \R^d$: \begin{enumerate}\renewcommand{\labelenumi}{(\roman{enumi})}
        \item for any $\kappa,\lambda=1,\dots,n$ \begin{multline*}
            f_\kappa^j(s,x) \partial_j (f_\lambda^i(s,\cdot) u_s)(x) - (f_{\kappa\lambda}')^i(s,x) \partial_i u_s(x) \overset{\alpha}{=}  f_\kappa^j(t,x) \partial_j (f_\lambda^i(t,\cdot) u_t)(x) - (f_{\kappa\lambda}')^i(t,x) \partial_i u_t(x) ;
        \end{multline*} and 
        \begin{multline*}
            f_\kappa^i(s,x) \partial_i u_s(x) - f_\kappa^i (t,x) \partial_i u_t(x)  \overset{2\alpha}{=}    \big(f_\kappa^j(t,x) \partial_j (f_\lambda^i(t,\cdot) \partial _iu_t(\cdot))(x) - (f_{\kappa\lambda}')^i(t,x) \partial_i u_t(x) \big) \delta W^\lambda_{s,t} ;
        \end{multline*}
        \item there is $u^\natural_{s,t}(x)$ such that $\sup_{x \in \R^d}|u^\natural_{s,t}(x)|\lesssim |t-s|^{3\alpha}$ and \begin{multline*}
            u_s(x) - u_t(x) = \int_s^t \underbrace{\frac{1}{2} a^{ij}(r,x) \partial^2_{ij}u_r(x) + b^i(r,x) \partial_i u_r(x)}_{= L_r u_r(x)} \, dr  + f^i_\kappa(t,x) \partial_i u_t(x) \, \delta W^\kappa_{s,t}  \\ + \big(f_\kappa^j(t,x) \partial_j (f_\lambda^i(t,\cdot) \partial_i u_t)(x)   - (f_{\kappa\lambda}')^i(t,x) \partial_i u_t(x) \big) \mathbb{W}^{\kappa \lambda}_{s,t} + u^\natural_{s,t}(x).
        \end{multline*}
    \end{enumerate}
    When the final-time condition $u_T = g \in C^0(\R^d;\R^d)$ is specified, we say that $u$ is a solution to \eqref{eq:backwardlinearroughPDE} with final-time condition $g$. 
\end{defn}

Equations of this type have been extensively studied in \cite{BFS24}. In particular, well posedness holds under the following \begin{assumptions} \label{assumptionsuniqueness_linearcase}
    Assume that $$[t \mapsto b(t,\cdot)] \in C^0([0,T]; C_b^2(\R^d;\R^d)) \cap \mathrm{Bd}([0,T];C^3_b(\R^d;\R^d)),$$ with similar conditions on $\sigma$, and assume that $(f,f') \in \mathscr{D}_W^{2\alpha}C_b^5(\R^d;\mathscr{L}(\R^n,\R^d)).$ 
\end{assumptions}

\begin{thm} \label{thm:existence_backwardlinearPDE}
    (see \cite[Proposition 4.8]{BFS24} and \cite[Theorem 4.9]{BFS24}) 
    Under \cref{assumptionsuniqueness_linearcase} and for any $g \in C^3_b(\R^d;\R)$, there exists a (unique)  solution $u$ to equation \eqref{eq:backwardlinearroughPDE} with final-time condition $g$. 
    Moreover, such a solution belongs to $C^{0,2}_b([0,T] \times \R^d; \R) \cap \mathrm{Bd}([0,T]; C^3_b(\R^d;\R))$ \footnote{$C^{0,2}_b([0,T] \times \R^d; \R)$ is the space of functions $g:[0,T] \times \R^d \to \R$ such that $g(t,\cdot) \in C^2(\R^d;\R)$ for any $t \in [0,T]$ and $  g, D_xg , D^2_x g \in C^0_b([0,T] \times \R^d;\R)$.} and the mapping  $[0,T] \ni t \mapsto u_t:= u(t,\cdot) \in C^2_b(\R^d;\R)$ is $\alpha$-H\"older continuous. 
\end{thm}

\begin{thm} \label{prop:duality}
    Let \cref{assumptionsuniqueness_linearcase} hold. 
    Then, for any $\nu \in \mathcal{P}_2(\R^d)$, there is at most one solution $\mu$ to \eqref{eq:solution_linearroughPDE} starting from $\nu$ such that 
    \begin{equation} \label{eq:Wassersteincondition}
        \sup_{\phi \in C^1,  |\phi|_{Lip} \le R} |\langle \mu_t- \mu_s , \phi \rangle | \overset{\alpha}{=} 0
        \qquad \text{for any $R>0$} \footnote{This condition implies that  $\mathcal{W}_1(\mu_t,\mu_s)\overset{\alpha}{=} 0$.
        Indeed, it is sufficient to take $R=1$ and to recall the Kantorovich duality theorem for 1-Wasserstein distances (see \cite[Corollary 5.4]{CARMONADELARUE_volumeI}), according to which $$\mathcal{W}_1(\mu_t,\mu_s) = \sup_{\phi \in C^1, |\phi|_{Lip} \le 1} \langle \mu_t- \mu_s , \phi \rangle . $$ } . 
    \end{equation}
\end{thm}

\begin{proof}
    Let $\mu^1,\mu^2$ be two solutions to \eqref{eq:solution_linearroughPDE} starting from $\nu \in \mathcal{P}_2(\R^d)$ and satisfying \eqref{eq:Wassersteincondition}. 
    We aim at showing that $\mu^1_\tau = \mu^2_\tau$ for any $\tau \in [0,T]$ by proving that $\langle \mu^1_\tau , \varphi \rangle = \langle \mu^2_\tau , \varphi \rangle$ for any $\varphi \in C^\infty_c(\R^d;\R)$. 
    Fix $\tau \in [0,T]$ and $\varphi \in C^\infty_c(\R^d;\R)$. 
    Under the assumptions of this theorem, we get the existence of a solution $u:[0,\tau] \times \R^d \to \R$ to the following rough PDE on $[0,\tau]$  \begin{equation} \label{eq:backwardroughPDE_duality}
        - du_t(x) = \left( \frac{1}{2}  a^{ij}(t,x) \partial^2_{ij} u_t(x)  + b^i(t,x) \partial_i u_t(x) \right) \, dt + f^i(t,x) \partial_i u_t(x) \, d\mathbf{W}_t  
    \end{equation}
    with final-time condition $u(\tau,\cdot) = \varphi$ (see \cref{thm:existence_backwardlinearPDE}).
    The conclusion follows if we prove that the map $[0,\tau] \ni t \mapsto \langle \mu^i_t, u_t \rangle = \int_{\R^d} u_t(x) \, \mu^i_t(dx)$ is constant ($i=1,2$). Indeed, we would have \begin{equation*}
        \langle \mu^1_\tau , \varphi \rangle = \langle \mu^1_\tau , u_\tau \rangle = \langle \mu^1_0 , u_0 \rangle = \langle \nu , u_0 \rangle = \langle \mu^2_0 , u_0 \rangle = \langle \mu^2_\tau , u_\tau \rangle = \langle \mu^2_\tau , \varphi \rangle .
    \end{equation*}
    The rest of the proof works for $\mu=\mu^i, \ i=1,2$. 
    We aim at proving that $[0,\tau] \ni t \mapsto \langle \mu_t, u_t \rangle \in \R$ is $(1+\varepsilon)$-H\"older continuous for some $\varepsilon >0$. 
    Let $s,t \in [0,\tau]$ with $s \le t$.  We can write \begin{equation*}
        \langle \mu_t, u_t \rangle - \langle \mu_s, u_s \rangle = \langle \mu_s, \delta u_{s,t} \rangle + \langle \delta \mu_{s,t}, u_s \rangle + \langle \delta \mu_{s,t}, \delta u_{s,t} \rangle .
    \end{equation*}
    From the fact that $u$ solves \eqref{eq:backwardroughPDE_duality} and $\mu_s$ is a probability measure on $\R^d$, it follows
    \begin{multline*}
        \langle \mu_s, \delta u_{s,t} \rangle = - \langle \mu_s , \int_s^t  L_r u_r \, dr \rangle - \langle \mu_s , f_\kappa^i(t,\cdot) \partial_i u_t \delta W^\kappa_{s,t} \rangle   \\ 
        - \langle \mu_s, f^i_\kappa(t,\cdot) \partial_i(f^j_\lambda(t,\cdot) \partial_j u_t) \mathbb{W}^{\kappa\lambda}_{s,t} \rangle  +  \langle \mu_s, (f_{\kappa\lambda}')^i(t,\cdot) \partial_i u_t \mathbb{W}^{\kappa\lambda}_{s,t} \rangle  + \langle \mu_s, u^\natural_{s,t} \rangle ,
    \end{multline*}
    where $|\langle \mu_s, u^\natural_{s,t} \rangle| \le \sup_{x \in \R^d} |u_{s,t}^\natural(x)| \overset{3\alpha}{=} 0$. 
    From the fact that $\mu$ solves \eqref{eq:solution_linearroughPDE} and $\sup_{t \in [0,\tau]} |u_t|_{C^3_b} < +\infty$ by \cref{thm:existence_backwardlinearPDE}, we have
    \begin{multline*}
        \langle \delta \mu_{s,t}, u_s \rangle = \int_s^t \langle \mu_r , L_r u_s \rangle \, dr + \langle \mu_s, f^i_\kappa (s,\cdot) \partial_i u_s \rangle \delta W^\kappa_{s,t}  \\ + \langle \mu_s , f_\kappa^i (s,\cdot) \partial_i( f^j_\lambda (s,\cdot) \partial_ju_s) \rangle \mathbb{W}^{\kappa\lambda}_{s,t} +  \langle \mu_s, (f_{\lambda\kappa}')^i (s,\cdot) \partial_i u_s \rangle \mathbb{W}_{s,t}^{\kappa\lambda} + \mu_{s,t}^{\natural,u_s} ,
    \end{multline*}
    with $|\mu_{s,t}^{\natural,u_s} | \overset{3\alpha}{=} 0$. 
    Notice how \begin{multline*}
        \int_s^t \langle \mu_r, L_r u_s \rangle \, dr - \langle \mu_s , \int_s^t  L_r u_r \, dr \rangle = \int_s^t \langle \mu_r, L_r (u_s-u_r) \rangle \, dr + \int_s^t \langle \mu_r - \mu_s ,   L_r u_r  \rangle \, dr \overset{1+\alpha}{=} 0,  
    \end{multline*}
    where the last equality follows from \cref{thm:existence_backwardlinearPDE} and noting that $x \mapsto L_t u_t(x)$ belongs to $C^1_b$ uniformly in $t$. 
    Moreover, being $u$ a solution to \eqref{eq:backwardroughPDE_duality} and by the assumptions on $(f,f')$, \begin{equation*}
        \langle \mu_s, f^i_\kappa (s,\cdot) \partial_i u_s  \rangle - \langle \mu_s, f^i_\kappa (t,\cdot) \partial_i u_t  \rangle \overset{2\alpha}{=} \langle \mu_s, f^i_\kappa (t,\cdot) \partial_i( f_\lambda^j(t,\cdot) \partial_j u_t) - (f'_{\kappa\lambda})^i(t,\cdot) \partial_i u_t \rangle \delta W^\lambda_{s,t}
    \end{equation*} and \begin{multline*}
        \langle \mu_s , f_\kappa^i (s,\cdot) \partial_i( f^j_\lambda (s,\cdot) \partial_ju_s) +(f_{\lambda\kappa}')^i (s,\cdot) \partial_i u_s \rangle -  \langle \mu_s, f^i_\kappa(t,\cdot) \partial_i(f^j_\lambda(t,\cdot) \partial_j u_t) - (f_{\kappa\lambda}')^i(t,\cdot) \partial_i u_t \rangle  \\ \overset{\alpha}{=} \langle \mu_s , (f'_{\kappa \lambda})^i (t,\cdot) \partial_i u_t \rangle +  \langle \mu_s , (f'_{\lambda \kappa})^i (t,\cdot) \partial_i u_t \rangle. 
    \end{multline*}
    Hence \begin{multline*}
        \langle \mu_s, \delta u_{s,t} \rangle + \langle \delta \mu_{s,t}, u_s \rangle \overset{3\alpha}{=} \langle \mu_s, f^i_\kappa (t,\cdot) \partial_i( f_\lambda^j(t,\cdot) \partial_j u_t) \rangle \delta W^\lambda_{s,t} \delta W^\kappa_{s,t}  \\ + \langle \mu_s , (f'_{\kappa \lambda})^i (t,\cdot) \partial_i u_t \rangle (\mathbb{W}^{\kappa \lambda}_{s,t} + \mathbb{W}^{\lambda \kappa}_{s,t} - \delta W^\kappa_{s,t} \delta W^\lambda_{s,t} ).
    \end{multline*}
    Recall that $\mathbf{W}$ is weakly geometric and, therefore, $\mathbb{W}^{\kappa\lambda}_{s,t} + \mathbb{W}^{\lambda\kappa}_{s,t} = \delta W^\lambda_{s,t} \delta W^\kappa_{s,t}$.
    Considering the regularity of $u$ coming from \cref{thm:existence_backwardlinearPDE}, together with condition (ii) of \cref{def:solution_meanfieldroughPDE} and the assumptions on $b,\sigma,(f,f')$,
    we can show that \begin{equation*}
        \begin{split}
            \langle \delta \mu_{s,t},  \delta u_{s,t} \rangle &= \int_s^t \langle \mu_r ,   L_r \delta u_{s,t}  \rangle \, dr + \langle \mu_s , f^i_\kappa(s,\cdot) \partial_i \delta u_{s,t} \rangle \delta W^\kappa_{s,t}  \\
            &\quad  + \langle \mu_s , f_\kappa^i (s,\cdot) \partial_i (f_\lambda^j (s,\cdot) \partial_j \delta u_{s,t}) + (f'_{\kappa\lambda})^i(s,\cdot) \partial_i \delta u_{s,t} \rangle \mathbb{W}^{\kappa\lambda}_{s,t} + \mu_{s,t}^{\natural,\delta u_{s,t}}  \\
            &\overset{3\alpha}{=} \langle \mu_s , f^i_\kappa(t,\cdot) \partial_i u_t - f^i_\kappa(s,\cdot) \partial_i u_s  + (f^i_\kappa(s,\cdot)- f^i_\kappa(t,\cdot)) \partial_i u_t \rangle \delta W^\kappa_{s,t}  \\
            &\overset{3\alpha}{=} - \langle \mu_s , f_\kappa^j (t,\cdot) \partial_j(f_\lambda^i(t,\cdot) \partial_i u_t) \rangle \delta W^\lambda_{s,t} \delta W^\kappa_{s,t}.   
        \end{split}
    \end{equation*}
    The statement is proved, since we showed that $\langle \mu_t , u_t \rangle - \langle \mu_s , u_s \rangle \overset{3\alpha}{=} 0$. 
\end{proof}

\section*{Conflict of Interest}
The authors declare that they have no conflict of interest.

\section*{Funding Declaration} 
FB is supported by Deutsche Forschungsgemeinschaft (DFG) - Project-ID 410208580 - IRTG2544 ``Stochastic Analysis in Interaction''.
PKF and WS acknowledge funding by Deutsche Forschungsgemeinschaft (DFG) - Project-ID 516748464 – CRC/TRR 388 ``Rough Analysis, Stochastic Dynamics and Related Fields''.

\printbibliography

\end{document}